\documentclass[11pt]{amsart}

\usepackage{amssymb, amscd}
\usepackage{epsfig, mathtools, tensor}
\usepackage[vmargin=1in, hmargin=1.25in]{geometry}
\usepackage[font=small,format=plain,labelfont=bf,up,textfont=it,up]{caption}
\usepackage{microtype}
\usepackage[shortlabels]{enumitem}
\setlist[itemize]{nosep}
\usepackage[backref=page, bookmarks, bookmarksdepth=2, colorlinks=true, linkcolor=blue, citecolor=blue, urlcolor=blue]{hyperref}
\usepackage{etoolbox}
\usepackage{tikz-cd}
\usepackage{tabularray}
\apptocmd{\thebibliography}{\raggedright}{}{}
\usepackage[only,llparenthesis, rrparenthesis, llbracket,rrbracket]{stmaryrd}
\SetSymbolFont{stmry}{bold}{U}{stmry}{m}{n}
\usepackage{bm}
\usepackage{refcount}

\mathtoolsset{showonlyrefs}

\newcommand{\Figure}[1]{%
  \par
  \nointerlineskip
  \vspace{5pt}%
  \noindent
  \makebox[\linewidth][c]{%
    \psfig{file=#1,scale=1}%
  }%
  \par
  \nointerlineskip
  \vspace{5pt}%
  \noindent\ignorespaces
}

\makeatletter
\patchcmd{\@maketitle}{\global\topskip42\p@\relax}
  {\global\topskip42\p@\relax \vspace*{-38pt}}
  {}{}
\makeatother

\setenumerate[0]{label=(\alph*)}

\renewcommand*{\backref}[1]{}
\renewcommand*{\backrefalt}[4]{%
    \ifcase #1 (Not cited.)%
    \or        (Cited on page~#2.)%
    \else      (Cited on pages~#2.)%
    \fi}

\DeclareSymbolFont{bbold}{U}{bbold}{m}{n}
\DeclareSymbolFontAlphabet{\mathbbold}{bbold}

\DeclareSymbolFont{extraup}{U}{zavm}{m}{n}
\DeclareMathSymbol{\varheartsuit}{\mathalpha}{extraup}{86}

\newcommand{\arxiv}[1]{\href{http://arxiv.org/abs/#1}{{\tt arXiv:#1}}}
\numberwithin{equation}{section}

\theoremstyle{plain}
\newtheorem{theorem}{Theorem}[section]
\newtheorem{maintheorem}{Theorem}

\newtheorem{lemma}[theorem]{Lemma}

\newtheorem{corollary}[theorem]{Corollary}

\newtheorem*{unnumberedclaim}{Claim}

\newtheorem{stepx}{Step}
\newenvironment{step}[1]
 {\renewcommand\thestepx{#1}\stepx}
 {\endstepx}

\newtheorem{claimx}{Claim}
\newenvironment{claim}[1]
 {\renewcommand\theclaimx{#1}\claimx}
 {\endclaimx}

\providecommand{\previous}{}



\providecommand{\previous}{}



\providecommand{\previous}{}



\theoremstyle{definition}
\newtheorem{asm}[theorem]{Assumption}

\newtheorem*{genusasm}{Genus assumption}

\newtheorem{defn}[theorem]{Definition}
\newenvironment{definition}[1][]{\begin{defn}[#1]\pushQED{\qed}}{\popQED \end{defn}}
\newtheorem{notn}[theorem]{Notation}

\newtheorem{covn}[theorem]{Convention}
\newenvironment{convention}[1][]{\begin{covn}[#1]\pushQED{\qed}}{\popQED \end{covn}}
\newtheorem{warg}[theorem]{Warning}
\newenvironment{warning}[1][]{\begin{warg}[#1]\pushQED{\qed}}{\popQED \end{warg}}

\theoremstyle{remark}
\newtheorem{rmk}[theorem]{Remark}
\newenvironment{remark}[1][]{\begin{rmk}[#1] \pushQED{\qed}}{\popQED \end{rmk}}
\newtheorem{eg}[theorem]{Example}
\newenvironment{example}[1][]{\begin{eg}[#1] \pushQED{\qed}}{\popQED \end{eg}}

\theoremstyle{plain}

\DeclareMathOperator{\Diff}{Diff}

\DeclareMathOperator{\Image}{Im}

\DeclareMathOperator{\GL}{GL}

\DeclareMathOperator{\PGL}{PGL}

\DeclareMathOperator{\PSL}{PSL}

\newcommand\R{\ensuremath{\mathbb{R}}}
\newcommand\C{\ensuremath{\mathbb{C}}}
\newcommand\Z{\ensuremath{\mathbb{Z}}}
\newcommand\Q{\ensuremath{\mathbb{Q}}}

\DeclareMathOperator{\HH}{H}
\newcommand\RH{\ensuremath{\widetilde{\HH}}}

\DeclareMathOperator{\RP}{\mathbb{RP}}

\DeclareMathOperator{\Aut}{Aut}

\newcommand\Span[1]{\ensuremath{\langle #1 \rangle}}
\newcommand\SpanSet[2]{\ensuremath{\langle \text{#1 $|$ #2} \rangle}}
\newcommand\Set[2]{\ensuremath{\left\{\text{#1 $|$ #2}\right\}}}

\newcommand\cC{\ensuremath{\mathcal{C}}}

\newcommand\cH{\ensuremath{\mathcal{H}}}

\newcommand\cM{\ensuremath{\mathcal{M}}}

\newcommand\cS{\ensuremath{\mathcal{S}}}
\newcommand\cT{\ensuremath{\mathcal{T}}}

\newcommand\fX{\ensuremath{\mathfrak{X}}}

\newcommand\bz{\ensuremath{\mathbf{z}}}

\newcommand\bbD{\ensuremath{\mathbb{D}}}

\newcommand\bbH{\ensuremath{\mathbb{H}}}

\newcommand\bbP{\ensuremath{\mathbb{P}}}

\newcommand\bbS{\ensuremath{\mathbb{S}}}

\newcommand\tA{\ensuremath{\widetilde{A}}}

\newcommand\tF{\ensuremath{\widetilde{F}}}

\newcommand\tU{\ensuremath{\widetilde{U}}}

\newcommand\tX{\ensuremath{\widetilde{X}}}

\newcommand\tZ{\ensuremath{\widetilde{Z}}}

\newcommand\tb{\ensuremath{\widetilde{b}}}

\newcommand\tf{\ensuremath{\widetilde{f}}}

\newcommand\tSigma{\ensuremath{\widetilde{\Sigma}}}

\newcommand\oG{\ensuremath{\overline{G}}}

\newcommand\oZ{\ensuremath{\overline{Z}}}

\newcommand\om{\ensuremath{\overline{m}}}

\newcommand\ow{\ensuremath{\overline{w}}}

\newcommand\Mod{\ensuremath{\operatorname{Mod}}}
\newcommand\PMod{\ensuremath{\operatorname{PMod}}}
\newcommand\Moduli{\ensuremath{\operatorname{\cM}}}
\newcommand\Teich{\ensuremath{\operatorname{\cT}}}
\newcommand\Hyp{\ensuremath{\operatorname{Hyp}}}
\newcommand\oModuli{\ensuremath{\operatorname{\overline{\cM}}}}
\newcommand\oTeich{\ensuremath{\operatorname{\overline{\cT}}}}
\newcommand\vcd{\ensuremath{\operatorname{vcd}}}
\newcommand\irr{\ensuremath{\operatorname{irr}}}
\newcommand\Gr{\ensuremath{\operatorname{Gr}}}
\newcommand\an{\ensuremath{\operatorname{an}}}
\newcommand\dz{\ensuremath{\operatorname{dz}}}
\newcommand\dw{\ensuremath{\operatorname{dw}}}
\newcommand\sym{\ensuremath{\operatorname{sym}}}

\title{The rational homology of quotients of the curve complex}

\author{Andrew Putman}
\address{Department of Mathematics; University of Notre Dame; 255 Hurley Hall; Notre Dame, IN 46556; USA}
\email{andyp@nd.edu}

\thanks{AP was supported by NSF grant DMS-2305183.}

\begin{document}
    
\newpage
        
\begin{abstract}
Let $\Mod_{g,n}$ be the mapping class group of a genus-$g$ surface with $n$ punctures
and let $\cC_{g,n}$ be its curve complex.
For a finite-index subgroup $G < \Mod_{g,n}$,
Boggi proved that $\RH_k(\cC_{g,n}/G;\Q) = 0$ for $g+n \gg k$.  This
plays an important role in the work of Putman--Wieland on the virtual first
Betti number of $\Mod_{g,n}$.  The proof of Boggi's theorem uses
mixed Hodge theory and
is embedded in his flawed paper purporting to show that the mapping class group
has the congruence subgroup property.  We give a detailed exposition of
Boggi's proof aimed at geometric topologists.
\end{abstract} 

\maketitle
\thispagestyle{empty}

\vspace{-20pt}
\section{Introduction}
\label{section:introduction}

Let $\Sigma_{g,n}$ be a connected oriented genus-$g$ surface with $n$ punctures.  We will sometimes
omit $n$ if $n=0$.  A simple closed curve on $\Sigma_{g,n}$ is {\em essential} if
it does not bound a disk or a once-punctured disk.  The {\em curve complex}
of $\Sigma_{g,n}$, denoted $\cC_{g,n}$, is the following simplicial complex:
\begin{itemize}
\item The vertices of $\cC_{g,n}$ are isotopy classes of essential simple closed curves on
$\Sigma_{g,n}$.
\item The $p$-simplices of $\cC_{g,n}$ are sets $\{\gamma_0,\ldots,\gamma_p\}$ of
distinct vertices that can be realized by disjoint curves.
\end{itemize}
The left-hand side of the following
figure shows curves on $\Sigma_{3,2}$ representing vertices of $\cC_{3,2}$, 
and the right-hand side shows the full subcomplex of $\cC_{3,2}$ spanned by these vertices:
\Figure{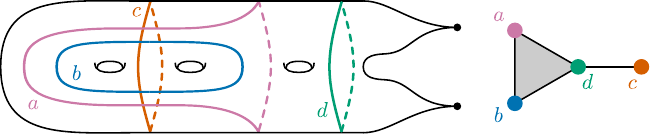}

\subsection{High connectivity}
The curve complex was introduced by Harvey \cite{HarveyCurve} to encode the combinatorial
structure at infinity of Teichm\"{u}ller space.  It was further developed by Harer,
who in \cite[Theorem 3.5]{HarerVCD} proved that $\cC_{g,n}$ is $c(g,n)$-connected 
for\footnote{Our convention is that a space $X$ is $c$-connected if for all $k \leq c$
and all maps $f\colon \bbS^k \rightarrow X$, there exists a map $F\colon \bbD^{k+1} \rightarrow X$
with $F|_{\bbS^k} = f$.
We have $\bbS^k = \bbD^k = \emptyset$ for $k \leq -1$ and $\bbD^0 = \{\text{pt}\}$, so this
means in particular that all spaces $X$ are $c$-connected for $c \leq -2$ and that $X$ is $(-1)$-connected
if and only if $X \neq \emptyset$.}
\[c(g,n) = \begin{cases}
2g-3 & \text{for $g \geq 1$ and $n = 0$},\\
2g+n-4 & \text{for $g \geq 1$ and $n \geq 1$},\\
n-5 & \text{for $g = 0$}.
\end{cases}\]
Here we are identifying the simplicial complex $\cC_{g,n}$ with its geometric realization, and we
make similar identifications throughout this paper.
See \cite{BrendleBroaddusPutman, HatcherVogtmannTethers, IvanovCerf, KentLeiningerSchleimer} 
for alternate approaches to various parts of Harer's theorem.

\subsection{Boggi's theorem}
Let $\Mod_{g,n}$ be the mapping class group of $\Sigma_{g,n}$, that is, the group of
isotopy classes of orientation-preserving diffeomorphisms of $\Sigma_{g,n}$.  These
diffeomorphisms are allowed to permute the punctures.  The group
$\Mod_{g,n}$ acts on $\cC_{g,n}$.  Consider a finite-index subgroup $G$ and the
corresponding quotient $\cC_{g,n}/G$.
Building on work of 
Arbarello--Cornalba \cite[Lemma 2.6]{ArbarelloCornalba},
Boggi \cite{BoggiCongruence} proved that $\RH_k(\cC_{g,n}/G;\Q)=0$ for
$k$ in the same range as in Harer's theorem:\footnote{We regard reduced 
homology as the homology of the augmented singular chain complex, with
the augmentation appearing in degree $-1$.  In particular,
$\RH_{-1}(\emptyset;\Q)=\Q$ and $\RH_k(\emptyset;\Q) = 0$ for $k \neq -1$.}

\begin{maintheorem}[{Boggi, cf.\ \cite[Lemma 5.5]{BoggiCongruence}}]
\label{maintheorem:boggi}
Fix $g,n \geq 0$.  Let $G < \Mod_{g,n}$ be a finite-index subgroup.
Then $\RH_k(\cC_{g,n}/G;\Q) = 0$ for $k \leq c(g,n)$.
\end{maintheorem}

We remark that the statement of \cite[Lemma 5.5]{BoggiCongruence} is different from 
Theorem \ref{maintheorem:boggi}.  The result \cite[Lemma 5.5]{BoggiCongruence} forms part of
Boggi's flawed proof that mapping class groups have the congruence subgroup property,
and the proof of \cite[Lemma 5.5]{BoggiCongruence} is exactly where that proof goes
wrong.  However, a close reading of the proof of \cite[Lemma 5.5]{BoggiCongruence} shows that it
does correctly establish Theorem \ref{maintheorem:boggi}.  Roughly speaking, Boggi's error
was trying to deduce integral results from Hodge theory.
See \cite{BoggiMathscinet} and \cite[\S 5]{LochakProfinite} for a more detailed discussion of this. 

\begin{remark}
Boggi's paper only handles finite-index subgroups $G < \Mod_{g,n}$ that
satisfy certain technical assumptions.  However, a standard and elementary argument
shows that the general case follows from this.  We will describe this later in the paper.
\end{remark}

\subsection{Why is Boggi's theorem interesting?}

As we will describe below, our primary motivation for caring about
Theorem \ref{maintheorem:boggi} is that it is a key ingredient in
the author's work with Wieland on the virtual first Betti number
of the mapping class group.  Before describing this, we point out
two other features of Theorem \ref{maintheorem:boggi}:
\begin{itemize}
\item It is one of the few unconditional results known about all finite-index
subgroups of the mapping class group. 
\item Though the statement of Theorem \ref{maintheorem:boggi} is
purely topological, its proof uses deep results about mixed
Hodge theory and the algebraic geometry of the Deligne--Mumford
compactification of the moduli space of curves.
\end{itemize}
The easiest nontrivial special case of Theorem \ref{maintheorem:boggi}
is as follows:
\begin{itemize}
\item Fix $g \geq 2$.  For all finite-index subgroups $G < \Mod_g$,
we have $\HH_1(\cC_g/G;\Q) = 0$.
\end{itemize}
It would be interesting to have a purely topological proof of even
this special case.  This seems entirely out of reach.

\subsection{Putman--Wieland conjecture}

For $g \geq 3$,
Ivanov \cite{IvanovQuestions} asked whether every finite-index subgroup of $\Mod_{g,n}$ has
first Betti number $0$.  With Wieland \cite{PutmanWieland}, the author proved that
this is roughly equivalent\footnote{The precise statement involves various shifts in $g$ and $n$, and
is carefully stated in the introduction of \cite{PutmanWieland}.} 
to what has become known as the Putman--Wieland conjecture.  
Since the precise statement of this conjecture is not needed for our paper,
we will only give an impressionistic version of it:
\begin{itemize}
\item Fix $g \geq 2$.  Let $p\colon \tSigma \rightarrow \Sigma_g$ be a finite regular branched cover and
let $v \in \HH_1(\tSigma;\Q)$ be nonzero.  Let $\Gamma$ be the subgroup of the mapping
class group of $\tSigma$ consisting of mapping classes that can be represented by
diffeomorphisms that normalize the deck group of $p\colon \tSigma \rightarrow \Sigma_g$.
Then the $\Gamma$-orbit of $v$ is infinite.
\end{itemize}
Markovi\'{c} \cite{MarkovicCounterexample} constructed a counterexample to the Putman--Wieland conjecture when $g=2$, but
it remains open for $g \geq 3$.  It turns out
to be connected to many other parts of mathematics and
has acquired a substantial literature.
See, for instance, \cite{BoggiPutmanSalter, GrunewaldLarsenLubotzkyMalestein, KlukowskiMarkovic, 
KrophollerFibering, LandesmanLitt1, LandesmanLitt2, LandesmanLitt3, LandesmanLittSawin, LittSurvey, LooijengaArithmetic, MarkovicTosic}.

\subsection{Relationship to Boggi's theorem}

The proof of the theorem in \cite{PutmanWieland} connecting the Putman--Wieland conjecture
to the vanishing of the virtual first Betti number of $\Mod_{g,n}$ uses
Theorem \ref{maintheorem:boggi} in an essential way.  This raises an obvious issue: the proof of
Theorem \ref{maintheorem:boggi} in \cite{BoggiCongruence} is part of a flawed
argument, and is also written for a reader who is fluent in mixed
Hodge theory.  Even for a reader who has the necessary background, it
is not easy to extract the proof from \cite{BoggiCongruence}.  

\subsection{Goal and audience of the paper}
Many people have told the author that it would be helpful to have a careful exposition
of Boggi's proof of Theorem \ref{maintheorem:boggi}.
Our goal is to give a self-contained account of the proof of Theorem \ref{maintheorem:boggi}
that is accessible to geometric topologists.  We envision our
reader as having roughly the background we did back in 2009--2010 when
we were working with Wieland on \cite{PutmanWieland}: 
expertise in the mapping class group and Teichm\"{u}ller theory,
but only a rudimentary knowledge of complex algebraic and analytic
geometry.

\begin{remark}
Topologists who learned about the moduli space of curves via Teichm\"{u}ller theory
typically think about it complex analytically.  As much as possible, we will give
definitions and constructions from this perspective.
However, the proof of Theorem \ref{maintheorem:boggi}
uses tools that only exist in the setting of algebraic (as opposed to analytic) geometry.
It will therefore be important for us that the spaces we construct
are actually algebraic varieties.  
\end{remark}

\begin{remark}
Since our goal is to make this paper accessible to topologists,
we include a substantial amount of background information.  In
particular, we will need many results about the moduli space of curves that are
well-known to experts, but for which we do not know appropriate
references that include complete proofs.  
We hope a reader with a more serious background
in algebraic geometry will pardon us for explaining so many basic
things.
\end{remark}

\subsection{Outline}

We start with two sections of preliminaries: in
\S \ref{section:mixedhodge} we survey
some basic facts about mixed Hodge theory,
and in \S \ref{section:normal} we survey some basic
facts about normalizations and the relationship between
algebraic and analytic geometry.
We then introduce the moduli space of curves and
its finite covers in \S \ref{section:modulicurves} -- \S \ref{section:modulicover}.
After that, we study the Deligne--Mumford compactification
in \S \ref{section:delignemumford} -- \S \ref{section:compactcover}.
We finally prove Theorem \ref{maintheorem:boggi} in \S \ref{section:boggiproof}.  

\subsection{Acknowledgements}

I would like to thank Autumn Kent, who first told me about Boggi's work at MSRI in 2007.
It was immediately clear to me that it should have implications for the virtual first
Betti number of $\Mod_{g,n}$.  I would also like to thank Ben Wieland.
In addition to collaborating with me to figure
out what those implications were, he spent a huge amount of time
explaining Boggi's arguments to me.  Any understanding I have of them comes
from him.  

I would finally like to thank the people who
have encouraged me to write an exposition of Boggi's proof of Theorem \ref{maintheorem:boggi}.  In
particular, Henry Wilton provided crucial encouragement in the
summer of 2026.

\subsection{Conventions}
Throughout this paper, all algebraic varieties are defined over $\C$.  We will
use the classical language of complex algebraic and analytic geometry rather than scheme
theory.  We will also frequently not distinguish between a simple closed curve on
a topological surface and its isotopy class, or between an orientation-preserving diffeomorphism
of such a surface and the mapping class it represents.

\section{The weight filtration}
\label{section:mixedhodge}

The proof of Theorem \ref{maintheorem:boggi} uses mixed Hodge theory, which
was introduced by Deligne in his classic papers \cite{DeligneI, DeligneII, DeligneIII}.
Only a small fragment of this theory is needed, namely the {\em weight filtration} on the cohomology
of a variety.  This section explains how the weight filtration works and
gives some simple examples of it.
This filtration is one of the reasons that the cohomology of an algebraic
variety is special compared to a general space or manifold:
\begin{itemize}
\item Many cohomology classes on a variety have a clear geometric origin, and maps between varieties
tend to respect that geometry in a rigid way.
\item Spectral sequences involving the cohomology of varieties often degenerate
at $E_2$.  We will sketch an example of how the weight filtration can be
used to prove such a degeneration result in \S \ref{section:simplenormalcrossing} below.
\end{itemize}
For proofs of everything we say, in addition to Deligne's
original papers we recommend \cite{PetersSteenbrink, GriffithsSchmid}.  We also recommend
the useful short survey \cite{Durfee}.

\begin{warning}
The weight filtration is structure on the cohomology of complex algebraic varieties, not complex
analytic varieties.  See Example \ref{example:badweight} below for an easy example illustrating
the necessity of working in the algebro-geometric setting.
\end{warning}

\begin{convention}
For a complex algebraic variety $X$, we will write $\HH^k(X;\Q)$ for the $k^{\text{th}}$
singular cohomology of $X$ endowed with its classical analytic topology (as opposed
to the Zariski topology).  Later in this paper it will be important to distinguish
the complex algebraic variety $X$ from its analytification, which is a complex
analytic variety that is denoted $X^{\an}$.  If $X$ is smooth, then $X^{\an}$
is a complex manifold.
\end{convention}

\subsection{Basic facts about the weight filtration}
\label{section:weightintro}

Let $X$ be a quasiprojective variety, that is, a Zariski open subspace of a projective variety.
Fix some $k \geq 0$, and let $H = \HH^k(X;\Q)$.
The weight filtration on $H$ is an increasing filtration
\[\begin{tikzcd}[column sep=tiny, cramped]
0 \arrow[r, phantom, "="] & W_{-1} H \arrow[r, phantom, "\subset"] & W_0 H \arrow[r, phantom, "\subset"] & \cdots \arrow[r, phantom, "\subset"] & W_{2k} H \arrow[r, phantom, "="] & H.
\end{tikzcd}\]
The associated graded of the weight filtration will be written
\[\Gr_q^W H = (W_q H) / (W_{q-1} H).\]
Roughly speaking, the weight filtration encodes the geometric origins of various cohomology classes:
some come from the smooth locus, some come from the singularities, and some come from the topology
at infinity.  Here are three key properties of it:
\begin{itemize}
\item[(I)] If $X$ is a smooth projective variety, then the weight filtration on $H = \HH^k(X;\Q)$
takes the simple form
\[\begin{tikzcd}[column sep=tiny, cramped]
0 \arrow[r, phantom, "="] & W_{k-1} H \arrow[r, phantom, "\subset"] & W_k H \arrow[r, phantom, "="] & H.
\end{tikzcd}\]
We say that $\HH^k(X;\Q)$ is {\em pure of weight $k$}.
\item[(II)] If $X$ is projective but singular, then the weight filtration on $H = \HH^k(X;\Q)$
takes the form
\[\begin{tikzcd}[column sep=tiny, cramped]
0 \arrow[r, phantom, "="] & W_{-1} H \arrow[r, phantom, "\subset"] & W_0 H \arrow[r, phantom, "\subset"] & \cdots \arrow[r, phantom, "\subset"] & W_{k} H \arrow[r, phantom, "="] & H.
\end{tikzcd}\]
Elements of $\Gr_q^W H$ for $q<k$ reflect the structure of the singularities.
\item[(III)] If $X$ is smooth but only quasiprojective, then the weight filtration on $H = \HH^k(X;\Q)$
takes the form
\[\begin{tikzcd}[column sep=tiny, cramped]
0 \arrow[r, phantom, "="] & W_{k-1} H \arrow[r, phantom, "\subset"] & W_k H \arrow[r, phantom, "\subset"] & \cdots \arrow[r, phantom, "\subset"] & W_{2k} H \arrow[r, phantom, "="] & H.
\end{tikzcd}\]
Elements of $\Gr_q^W H$ for $q>k$ reflect the geometry at infinity.
\end{itemize}
We will give some examples below after discussing the strong sense in which
the weight filtration is preserved by algebraic maps.

\subsection{Strict preservation of the weight filtration}
Let $X$ and $Y$ be quasiprojective varieties and let $f\colon X \rightarrow Y$ be an algebraic
map.  The induced map $f^{\ast}\colon \HH^k(Y;\Q) \rightarrow \HH^k(X;\Q)$ preserves the weight
filtrations in the sense that for all $0 \leq q \leq 2k$ we have
\[f^{\ast}(W_q \HH^k(Y;\Q)) \subset W_q \HH^k(X;\Q).\]
In fact, even more is true: the weight filtration is {\em strictly preserved}, i.e., for
all $0 \leq q \leq 2k$ we have
\[f^{\ast}(W_q \HH^k(Y;\Q)) = \Image(f^{\ast}) \cap W_q \HH^k(X;\Q).\]
This has the following consequence.  Say that a subspace $V$ of $\HH^k(X;\Q)$
is {\em pure of weight $q$} if 
\[V \subset W_q \HH^k(X;\Q) \quad \text{and} \quad V \cap W_{q-1} \HH^k(X;\Q) = 0.\]
We then have:

\begin{lemma}
\label{lemma:strictpreserve}
Let $X$ be a quasiprojective variety, let $Y$ be a smooth projective variety, and
let $f\colon X \rightarrow Y$ be an algebraic map.  Then 
the image of the map 
$f^{\ast}\colon \HH^k(Y;\Q) \rightarrow \HH^k(X;\Q)$
is pure of weight $k$.
\end{lemma}
\begin{proof}
For some $v \in \HH^k(Y;\Q)$, assume that $f^{\ast}(v) \in W_{k-1} \HH^k(X;\Q)$.  We must prove
that $f^{\ast}(v)=0$.  To see this, note that since the weight filtration is strictly preserved
there must exist some $v' \in W_{k-1} \HH^k(Y;\Q)$ such that $f^{\ast}(v') = f^{\ast}(v)$.
Since $\HH^k(Y;\Q)$ is pure of weight $k$, we must have $v' = 0$.  The lemma follows.
\end{proof}

\begin{remark}
Other kinds of geometrically defined maps between the cohomology groups of 
quasiprojective varieties tend to
strictly preserve the weight filtration, though often with a weight shift encoded
in what is called a ``Tate twist''.  For instance, this holds for boundary maps
in long exact sequences, Gysin maps, etc.  We will not delve too deeply into
this, but such results are a key input to deeper applications of mixed
Hodge theory.
\end{remark}

\subsection{Examples}
We now give some simple examples illustrating how the weight filtration works.

\begin{example}
\label{example:p1minus}
Fix $n \geq 2$.  Let $p_1,\ldots,p_n \in \bbP^1$ be distinct points, and consider
$X = \bbP^1 \setminus \{p_1,\ldots,p_n\}$.  We have $\HH^1(X;\Q) \cong \Q^{n-1}$, and
it turns out that the weight filtration on $H = \HH^1(X;\Q)$ is
\[\begin{tikzcd}[column sep=tiny, cramped]
0 \arrow[r, phantom, "="] & W_{1} H \arrow[r, phantom, "\subset"] & W_2 H \arrow[r, phantom, "="] & H.
\end{tikzcd}\]
In other words, $\HH^1(X;\Q)$ is pure of weight $2$.  
As an important special case, $\HH^1(\C^{\times};\Q) = \Q$ is pure of weight $2$.
Elements of $\HH^1(X;\Q) \cong \Q^{n-1}$ reflect what is happening ``at infinity'', i.e.,
near $\bbP^1 \setminus X = \{p_1,\ldots,p_n\}$.  That $\HH^1(X;\Q)$ is pure of weight
$2$ therefore illustrates (III) from \S \ref{section:weightintro}.
\end{example}

\begin{example}
\label{example:curveminus}
Fix $n \geq 1$ and $g \geq 1$.  Let $C$ be a smooth genus-$g$ projective curve.  As a topological space,
$C$ is homeomorphic to $\Sigma_g$.  Since $C$ is smooth and projective,
$\HH^1(C;\Q) \cong \Q^{2g}$ is pure of weight $1$.  Let $p_1,\ldots,p_n \in C$ be distinct points, and
consider $X = C \setminus \{p_1,\ldots,p_n\}$.  We have $\HH^1(X;\Q) \cong \Q^{2g+n-1}$, and
$\HH^1(X;\Q)$ fits into a short exact sequence
\[0 \rightarrow \HH^1(C;\Q) \rightarrow \HH^1(X;\Q) \rightarrow \Q^{n-1} \rightarrow 0\]
whose first map is the one induced by the inclusion $X \hookrightarrow C$.  Since
$C$ is a smooth projective variety, Lemma \ref{lemma:strictpreserve} implies that
the image of $\HH^1(C;\Q)$ in $\HH^1(X;\Q)$ is pure of weight $1$.  It turns
out that the weight filtration of $H = \HH^1(X;\Q)$ is
\[\begin{tikzcd}[column sep=tiny, cramped]
0 \arrow[r, phantom, "="] & W_{0} H \arrow[r, phantom, "\subset"] & W_1 H \arrow[r, phantom, "\subset"] & W_2 H \arrow[r, phantom, "="] & H 
\end{tikzcd}\]
with $\Gr_1^W H \cong \HH^1(C;\Q)$ and $\Gr_2^W H \cong \Q^{n-1}$.
This also illustrates (III) from \S \ref{section:weightintro}.
\end{example}

\begin{example}
\label{example:weight3}
Let $X$ be a projective curve of the form $X = X_1 \cup X_2$, where:
\begin{itemize}
\item $X_1$ is a smooth genus-$g$ algebraic curve with $g \geq 1$; and
\item $X_2$ is isomorphic to $\bbP^1$; and
\item $X_1$ and $X_2$ intersect transversely in two points $p$ and $q$.
\end{itemize}
This is illustrated here:
\Figure{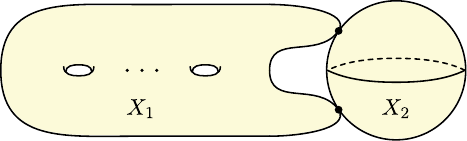}
We have
\begin{align*}
\HH^1(X_1;\Q) \cong &\Q^{2g},\\
\HH^1(X_2;\Q) \cong &0,\\
\RH^0(X_1 \cap X_2;\Q) = &\RH^0(\{p,q\};\Q) \cong \Q.
\end{align*}
The cohomology group $\HH^1(X;\Q) \cong \Q^{2g+1}$ fits into a Mayer--Vietoris exact sequence
\[\begin{tikzcd}
0 \arrow{r} & \RH^0(X_1 \cap X_2;\Q) \arrow{r} & \HH^1(X;\Q) \arrow{r} & \HH^1(X_1;\Q) \oplus \HH^1(X_2;\Q) \arrow{r} & 0.
\end{tikzcd}\]
Here $\RH^0(X_1 \cap X_2;\Q)$ and $\HH^1(X_1;\Q)$ are pure of weights $0$ and $1$, respectively.
It turns out that all these maps strictly preserve the weight filtrations, so
the weight filtration of $H=\HH^1(X;\Q)$ is
\[\begin{tikzcd}[column sep=tiny, cramped]
0 \arrow[r, phantom, "="] & W_{-1} H \arrow[r, phantom, "\subset"] & W_0 H \arrow[r, phantom, "\subset"] & W_1 H \arrow[r, phantom, "="] & H 
\end{tikzcd}\]
with $\Gr_0^W H \cong \RH^0(X_1 \cap X_2;\Q)$ and $\Gr_1^W H \cong \HH^1(X_1;\Q)$.
This illustrates (II) from \S \ref{section:weightintro}.
\end{example}

\begin{example}
\label{example:badweight}
Using the examples above, we can illustrate how important it is that
we are working in the algebraic category rather than the analytic one.
Let $E$ be an elliptic curve.  As a complex analytic
variety, we can identify $E$ with  $\C/\Lambda$ for a lattice $\Lambda$.  Let
$\pi\colon \C \rightarrow E$ be the projection.
Pick some nonzero $\lambda \in \Lambda$, and
define an analytic map $f\colon \C^{\times} \rightarrow E$ via
\[f(e^z) = \pi(\frac{\lambda}{2\pi i} z) \quad \text{for $z \in \C$}.\]
The map $f^{\ast}\colon \HH^1(E;\Q) \rightarrow \HH^1(\C^{\times};\Q)$ is nonzero; however,
$\HH^1(E;\Q)$ is pure of weight $1$ since $E$ is a smooth projective variety and $\HH^1(\C^{\times};\Q)$
is pure of weight $2$ (Example \ref{example:p1minus}).  It
follows that $f^{\ast}$ does {\em not} preserve the weight filtration.
\end{example}

\subsection{Weight filtration on a simple normal crossing variety}
\label{section:simplenormalcrossing}

For our proof of Theorem \ref{maintheorem:boggi}, we will need a generalization
of Example \ref{example:weight3}.  A {\em projective simple normal crossing variety}
of dimension $d$ is a projective variety $X$ such that for some $n \geq 1$ we have the following:
\begin{itemize}
\item $X = X_1 \cup \cdots \cup X_n$, where the $X_i$ are distinct 
smooth irreducible $d$-dimensional projective varieties.
The $X_i$ are called the {\em components} of $X$.
\item The $X_i$ intersect transversely in the following sense.  Consider some $x \in X$.  Let
$i_0 < \cdots < i_p$ be the indices such that $x \in X_{i_j}$.  Then $0 \leq p \leq d$, and
there exist complex analytic neighborhoods $U$ of $x$ in $X$ and $V$ of the origin in
\[\Set{$(z_0,\ldots,z_{d}) \in \C^{d+1}$}{$z_0 \cdots z_p = 0$}\]
and a complex analytic isomorphism $U \cong V$.  For $0 \leq j \leq p$, this
isomorphism takes $X_{i_j} \cap U$ to
$\Set{$(z_0,\ldots,z_d) \in V$}{$z_j = 0$}$.
\end{itemize}
Let $X$ be a projective simple normal crossing variety with components
$X_1,\ldots,X_n$.  Let
\[X^{[p]} = \bigsqcup_{i_0 < \cdots < i_p} X_{i_0} \cap \cdots \cap X_{i_p} \quad \text{for $p \geq 0$}.\]
Since the $X_i$ intersect transversely, this is a (possibly empty) disjoint union of smooth projective
varieties of dimension $d-p$.  In particular, $X^{[p]} = \emptyset$ for
$p>d$.  We can compute the cohomology of $X$ using the
Mayer--Vietoris spectral sequence
\[E^{pq}_1 = \HH^q(X^{[p]};\Q) \Rightarrow \HH^{p+q}(X;\Q).\]
Each $\HH^q(X^{[p]};\Q)$ is pure of weight $q$.  The $d_1$-differentials
are of the form
\[d_1\colon \HH^q(X^{[p]};\Q) \rightarrow \HH^q(X^{[p+1]};\Q).\]
This differential is the alternating sum of restriction maps associated to
inclusions of intersections.  Since the domain and codomain here are both pure of weight $q$, it is possible
for the $d_1$-differentials to be nonzero.  

However, the higher differentials
$d_r\colon E^{pq}_r \rightarrow E^{p+r,q-r+1}_r$ 
connect terms that are pure of different weights.  Using this,\footnote{Plus some
additional foundational work -- Deligne sets up an abelian category of mixed Hodge structures
and proves that this is a spectral sequence of mixed Hodge structures.} Deligne proved
that the above spectral sequence degenerates at $E_2$.  
For each $k \geq 0$, the above spectral sequence computes the weight filtration of $\HH^k(X;\Q)$
in the following sense:
\[\Gr_q^W \HH^k(X;\Q) \cong E^{k-q,q}_{\infty} = E^{k-q,q}_2.\]

\subsection{Dual complex of a simple normal crossing variety}
\label{section:dualcomplex}

Continue to let $X$ be a projective simple normal crossing variety
with irreducible components $X_1,\ldots,X_n$.  The {\em dual complex} of $X$, denoted $\Delta(X)$,
is the semisimplicial set whose set of $p$-simplices is $\pi_0(X^{[p]})$.  In other
words, the $p$-simplices of $\Delta(X)$ are the connected components of $X^{[p]}$
and the face maps correspond to inclusions of connected components.

In the spectral sequence 
\[E^{pq}_1 = \HH^q(X^{[p]};\Q) \Rightarrow \HH^{p+q}(X;\Q),\]
the $q=0$ row is exactly the semisimplicial cochain complex for $\Delta(X)$.  Since
this spectral sequence degenerates at $E_2$ and computes the weight
filtration on each $\HH^k(X;\Q)$, we deduce:

\begin{theorem}
\label{theorem:weight0}
Let $X$ be a projective simple normal crossing variety.  For each $k \geq 0$,
we have $W_0 \HH^k(X;\Q) \cong \HH^k(\Delta(X);\Q)$.
\end{theorem}

\begin{remark}
This implies that $\HH^k(\Delta(X);\Q)$ is isomorphic to a subspace of
$\HH^k(X;\Q)$, which is not obvious at all.
\end{remark}

\begin{example}
We revisit Example \ref{example:weight3} from this point of view.
In that example, $X$ is a projective curve of the form $X = X_1 \cup X_2$, where:
\begin{itemize}
\item $X_1$ is a smooth genus-$g$ algebraic curve with $g \geq 1$; and
\item $X_2$ is isomorphic to $\bbP^1$; and
\item $X_1$ and $X_2$ intersect transversely in two points $p$ and $q$.
\end{itemize}
The curve $X$ is a simple normal crossing variety with irreducible components $X_1$ and $X_2$.
Its dual complex $\Delta(X)$ is the semisimplicial set with:
\begin{itemize}
\item two vertices $X_1$ and $X_2$; and
\item two edges, namely the components $p$ and $q$ of $X_1 \cap X_2$.  
\end{itemize}
Both edges connect $X_1$ to $X_2$, so
the geometric realization of $\Delta(X)$ is a circle.\footnote{Since both edges connect the same two vertices, the
semisimplicial set $\Delta(X)$ is not a simplicial complex.}  By Theorem \ref{theorem:weight0}, we have
\[W_0 \HH^1(X;\Q) \cong \HH^1(\Delta(X);\Q) \cong \HH^1(\bbS^1;\Q) = \Q,\]
just like we computed in Example \ref{example:weight3}.  
\end{example}

\section{Some algebraic geometry: normalization and algebraization}
\label{section:normal}

This section covers some preliminary results about algebraic geometry.
The two main topics are normalizations of varieties and the relationship
between algebraic and analytic varieties.
See \cite[Chapter III, \S 8]{MumfordRed} for a textbook introduction to
normalization, and see \cite[Appendix B]{Hartshorne} for a short textbook introduction to
the relationship between algebraic and analytic varieties.\footnote{One result we
will discuss is the generalized Riemann existence theorem.  Be warned
that the statement of this in
\cite[Appendix B]{Hartshorne} is missing the important hypothesis that
the finite map is a finite unbranched cover.  See Theorem \ref{theorem:generalizedriemann}
below for the correct statement.}

\begin{remark}
In this section, we will be careful to distinguish in our notation a complex
algebraic variety $X$ from its analytification $X^{\an}$.  We will
use letters like $X,Y,\ldots$ for algebraic varieties and
letters like $M,N,\ldots$ for analytic varieties. 
\end{remark}

\subsection{Algebraic vs analytic properties}
\label{section:algebravsanalysis}

Let $X$ be an algebraic variety.  Many standard properties of varieties
hold for $X$ if and only if they hold for its analytification $X^{\an}$.
Here are some examples:
\begin{itemize}
\item[(i)] $X$ is connected if and only if $X^{\an}$ is connected \cite[Expos\'{e} XII, Proposition 2.4]{SGA1}.
\item[(ii)] $X$ is irreducible if and only if $X^{\an}$ is irreducible \cite[Expos\'{e} XII, Proposition 2.4]{SGA1}.
\item[(iii)] $X$ is smooth if and only if $X^{\an}$ is smooth \cite[Expos\'{e} XII, Proposition 3.1]{SGA1}.
\end{itemize}
In some settings, there is little difference between the algebraic and analytic categories.
Here is one result of this nature.
An analytic variety $M$ is said to be {\em proper} if it is compact, and an algebraic
variety $X$ is proper if $X^{\an}$ is proper.\footnote{For the equivalence between this definition
and the usual algebraic definition of being proper, see \cite[Expos\'{e} XII, Proposition 3.2]{SGA1}.}
For instance, projective varieties are proper.  We then have:
\begin{itemize}
\item If $X$ and $Y$ are proper algebraic varieties, then a map $X \rightarrow Y$ is algebraic
if and only if the corresponding map $X^{\an} \rightarrow Y^{\an}$ is analytic \cite[Expos\'{e} XII, Corollaire 4.5]{SGA1}.  
\end{itemize}
This is false in the non-proper setting; for instance, the analytic map $f\colon \C \rightarrow \C$
defined by $f(z) = e^z$ is not algebraic.  We must be careful to avoid such pathologies.

\subsection{Normal varieties}
\label{section:normaldef}

Our first main topic in this section is a weakening of being smooth called being
normal.

Let $M$ be a complex analytic variety,
let $U \subset M$ be an open set, and let $A \subsetneq U$ be a closed analytic
subvariety.  An analytic function $f\colon U \setminus A \rightarrow \C$
is {\em locally bounded} near $A$ if for all $a \in A$, there exists an open neighborhood
$V \subset U$ of $a$ such that $f|_{V \setminus A}$ is bounded.  The generalized
Riemann extension theorem says that if $U$ is smooth, then all analytic
functions $f\colon U \setminus A \rightarrow \C$ that are locally bounded near $A$ extend
to analytic functions on $U$.  

An irreducible analytic variety $M$ is said to be normal if it
satisfies the conclusion of the Riemann extension theorem.  More generally:

\begin{definition}
An analytic variety $M$ is {\em normal} if all connected components of $M$
are irreducible and for all connected open sets $U \subset M$ and all closed
analytic subvarieties $A \subsetneq U$
we have the following:
\begin{itemize}
\item All analytic functions $f\colon U \setminus A \rightarrow \C$ that are
locally bounded near $A$ extend to analytic functions on $U$.
\end{itemize}
An algebraic variety $X$ is {\em normal} if its analytification $X^{\an}$ is normal.
\end{definition}

\begin{remark}
The usual definition of normality is much more algebraic.  See \cite[Chapter II, \S 7]{Demailly}
for a proof that it is equivalent to the above definition in the analytic
category, and \cite[Expos\'{e} XII, Proposition 2.1]{SGA1} for a proof that
it is equivalent to the above definition in the algebraic category.
\end{remark}

The following examples illustrate the fact that a variety is normal if
its singularities are not too extreme:

\begin{example}
The generalized Riemann extension theorem says that smooth algebraic and analytic
varieties are normal.
\end{example}

\begin{example}
\label{example:notnormal}
Let $X = \Set{$(x,y) \in \C^2$}{$y^2=x^2(x+1)$}$.  Then $X$ is not normal.  Indeed,
the analytic function $f\colon X \setminus 0 \rightarrow \C$ defined
by
\[f(x,y) = y/x \quad \text{for $(x,y) \in X \setminus 0$}\]
is locally bounded near $0$ but does not extend continuously over $0$ since
for a small real $t \neq 0$ we have the points $(t,\pm t\sqrt{1+t})$ of $X \setminus 0$, and
\[\lim_{t \rightarrow 0} f(t,\pm t\sqrt{1+t}) = \lim_{t \rightarrow 0} \pm \sqrt{1+t} = \pm 1.\]
It follows that $X$ is not normal.
\end{example}

An algebraic or analytic curve is normal if and only if
it is smooth.  In higher dimensions, a normal variety can have
singularities.  One source of such examples is the following:

\begin{lemma}
\label{lemma:normalquotient}
Let $M$ be a normal analytic variety and let $G$ be a finite group acting
on $M$ by analytic automorphisms.  Then $M/G$ is normal.
\end{lemma}
\begin{proof}
Each connected component of $M/G$ is the image of a connected component of $M$.
Since the connected components of $M$ are irreducible, the connected components
of $M/G$ are also irreducible.
Let $p\colon M \rightarrow M/G$ be the projection.  Consider a connected open
set $U \subset M/G$, a closed analytic subvariety $A \subsetneq U$, and an
analytic function $f\colon U \setminus A \rightarrow \C$ that
is locally bounded near $A$.  Our goal is to extend $f$ to an analytic
function $F\colon U \rightarrow \C$.

Set $\tU = p^{-1}(U)$ and $\tA = p^{-1}(A)$.
Let $\tf\colon \tU \setminus \tA \rightarrow \C$ be
$\tf = f \circ p$.  Consider a component $\tU_i$ of $\tU$.  Since $p$ is both
an open and a closed map and $U$ is connected, $p(\tU_i) = U$.  Let 
$\tA_i = \tU_i \cap \tA$, so $\tA_i$ is a proper closed analytic subvariety of $\tU_i$.
Set $\tf_i = \tf|_{\tU_i \setminus \tA_i}$.
The function $\tf_i$ is locally bounded near $\tA_i$, so
since $M$ is normal the function $\tf_i$ extends to an analytic 
function $\tF_i\colon \tU_i \rightarrow \C$.  Applying this
to each component of $\tU$, we get an extension $\tF\colon \tU \rightarrow \C$ of
$\tf$.  The function $\tf$ is $G$-invariant,
so by continuity so is $\tF$.  It follows that $\tF$ descends to
an analytic function $F\colon U \rightarrow \C$ extending $f$.
\end{proof}

\begin{example}
\label{example:rpcone}
Let $C_2$ be the cyclic group of order $2$.  For
$n \geq 1$, let $C_2$ act on $\C^n$ as multiplication by $-1$.
Since $\C^n$ is smooth it is normal, so Lemma \ref{lemma:normalquotient} implies 
that $\C^n/C_2$ is normal.  The variety $\C^n/C_2$ is not smooth for $n \geq 2$
since as a topological space it is homeomorphic to the open cone on $\RP^{2n-1}$.
\end{example}

\begin{remark}
\label{remark:smoothcriterion}
Just like in Example \ref{example:rpcone}, let $M$ be a normal analytic variety and let $G$ be a finite group acting on
$M$ by analytic automorphisms.  Let $p\colon M \rightarrow M/G$ be the quotient.  Lemma \ref{lemma:normalquotient} says that
$M/G$ is normal.  Consider a point $x \in M$ at which $M$ is smooth.  How can we tell if $M/G$ is smooth at $p(x)$?

Let $G_x<G$ be the stabilizer of $x$ and $W = T_x M$ be the tangent space of $M$ at $x$, so $G_x \circlearrowright W$.  
Let $y \in W/G_x$ be the image of $0 \in W$.  
The point $p(x)$ has an analytic neighborhood that is isomorphic to a neighborhood of $y$,
so the question is equivalent to determining if $y$ is a smooth point of 
$W/G_x$.  

This has a simple answer: $y$ is a smooth point of $W/G_x$
if and only if the image of $G_x$ in $\GL(W)$ is generated by pseudo-reflections, that is, elements whose fixed
subspace has codimension $1$.  See Lemma \ref{lemma:notmanifold} and Remark \ref{remark:notmanifold} for more details.
This fact explains why we needed $n \geq 2$ in Example \ref{example:rpcone}.
\end{remark}

\subsection{Proper maps}

Our next goal is to discuss the process of normalization of an algebraic
variety, which requires some preliminary definitions.
If $f\colon M \rightarrow N$ is an analytic map between analytic varieties, then
$f$ is said to be {\em proper} if it is proper in the sense of point-set topology: preimages
of compact sets are compact.  An algebraic map $f\colon X \rightarrow Y$ between algebraic varieties
is {\em proper} if the induced map $f^{\an}\colon X^{\an} \rightarrow Y^{\an}$ is proper.

For an algebraic map $f\colon X \rightarrow Y$ between algebraic varieties, the usual definition
of being proper is that $f$ is universally closed.  This is a strengthening of being closed, that is,
taking Zariski closed subsets of $X$ to Zariski closed subsets of $Y$.  See
\cite[Expos\'e XII, Proposition 3.2]{SGA1} for a proof that it is equivalent to the definition
we gave.  To illustrate the meaning of our definition, we prove that it does imply
that $f$ is closed:

\begin{lemma}
\label{lemma:properclosed}
Let $f\colon X \rightarrow Y$ be a proper map between algebraic varieties and let $Z \subset X$
be a closed subvariety.  Then $f(Z) \subset Y$ is a closed subvariety.
\end{lemma}
\begin{proof}
A classical theorem of Chevalley (see, e.g., \cite[\S I.8, Corollary 2]{MumfordRed}) says
that $f(Z)$ is constructible, that is, a finite union of sets
of the form $U \cap C$ with $U$ Zariski open and $C$ Zariski closed.  
Since $f^{\an}\colon X^{\an} \rightarrow Y^{\an}$ is proper in the sense of point-set topology,
it takes closed sets to closed sets.  It follows that $f(Z)$ is closed in the analytic
topology on $Y$.  We can now appeal to a standard result
in algebraic geometry (see, e.g., \cite[Expos\'e XII, Corollary 2.3]{SGA1} or \cite[Theorem 5.4.4]{ArtinBook}) saying
that a constructible subspace of an algebraic variety that is closed in
the analytic topology is closed in the Zariski topology.
\end{proof}

\begin{remark}
It is also true that a proper map $f\colon M \rightarrow N$ between analytic
varieties takes analytic subvarieties to analytic subvarieties, but this
is a difficult theorem of Remmert.  See \cite[Theorem II.8.8]{Demailly}.
\end{remark}

\subsection{Finite maps}
\label{section:finitemaps}

If $f\colon M \rightarrow N$ is an analytic map between analytic varieties,
then $f$ is said to be {\em finite} if $f$ is proper and has finite fibers.
If $f$ is finite and surjective, it can be viewed as a finite ramified cover.
An algebraic map $f\colon X \rightarrow Y$ between algebraic varieties is {\em finite} if the induced map
$f^{\an}\colon X^{\an} \rightarrow Y^{\an}$ is finite.  Finite maps not only take
Zariski closed sets to Zariski closed sets, but also preserve dimension: if $Z \subset X$
is an irreducible closed subvariety, then $f(Z) \subset Y$ is also an irreducible
closed subvariety with $\dim Z = \dim f(Z)$ (see, e.g., \cite[\S 8, Theorem 3]{MumfordRed}).

\begin{remark}
The usual definition for a map between algebraic varieties to be finite is
purely algebraic.  See \cite[Expos\'e XII, Proposition 3.2]{SGA1} for
a proof that it is equivalent to the analytic definition above.
\end{remark}

\subsection{Normalization}
\label{section:normalization}

Let $X$ be an irreducible algebraic variety and let $U \subset X$
be a nonempty Zariski open subset.  Let $p\colon \tU \rightarrow U$ be a finite
surjective map with $\tU$ normal and irreducible.  A {\em normalization} of
$X$ with respect to $p\colon \tU \rightarrow U$ is an irreducible normal algebraic
variety $\tX$ containing $\tU$ as a Zariski open subset together
with an extension of $p$ to a finite map $p\colon \tX \rightarrow X$.

Normalizations always exist, and are unique in the sense that if
$p'\colon \tX' \rightarrow X$ is another normalization of $X$
with respect to $p\colon \tU \rightarrow U$ then there exists
an isomorphism $\phi\colon \tX \rightarrow \tX'$ restricting to the
identity on $\tU$ such that the following diagram commutes:
\[\begin{tikzcd}
\tX \arrow{rr}{\phi} \arrow{dr}{p} & & \tX' \arrow{dl}[swap]{p'} \\
& X &
\end{tikzcd}\]
Here are some examples:

\begin{example}
\label{example:ordinarynormalization}
Let $X$ be an irreducible algebraic variety and let $U \subset X$ be the smooth
locus of $X$.  The normalization of $X$ with respect to the identity map
$U \rightarrow U$ is often just called the normalization of $X$.  It
is an irreducible normal variety $\tX$ containing $U$ as a Zariski
open subset equipped with a finite map $p\colon \tX \rightarrow X$ restricting
to the identity on $U$.  We will give a concrete example of this
in the next example.  

If $X$ is not normal, then the singularities
of $\tX$ are simpler than those of $X$.  Indeed, if $X$ is an irreducible algebraic
curve, then $\tX$ is a smooth irreducible algebraic curve and
$p\colon \tX \rightarrow X$ is a resolution of the singularities of $X$.
For higher dimensional varieties $X$, taking the normalization is the first
step in resolving the singularities of $X$.
\end{example}

\begin{example}
Here is a concrete example of the normalization of a variety.
Let $X = \Set{$(x,y) \in \C^2$}{$y^2=x^2(x+1)$}$.  As we noted
in Example \ref{example:notnormal}, the algebraic curve $X$
is not normal.  The normalization of $X$ is the smooth curve
$\C$, which is equipped with the finite map $p\colon \C \rightarrow X$
defined by
\[p(z) = (z^2-1,z(z^2-1)) \in X \quad \text{for all $z \in \C$}.\qedhere\]
\end{example}
 
\begin{example}
Let $X = \C$ and let $U = X \setminus \{0,1\}$.  Set
$\tU = \Set{$(x,y) \in U \times \C$}{$y^2 = x(x-1)$}$, and
let $p\colon \tU \rightarrow U$ be the projection onto the
first coordinate.  The map $p\colon \tU \rightarrow U$
is a $2$-fold covering space map, and $\tU$ is smooth and hence
normal.  The normalization of $X$ with respect to
$p\colon \tU \rightarrow U$ is the smooth variety
$\tX = \Set{$(x,y) \in X \times \C$}{$y^2 = x(x-1)$}$, equipped
with the projection $p\colon \tX \rightarrow X$ onto
the first coordinate.  The map $p$ is not an unbranched 
covering space map since it is ramified over $0$ and $1$.
\end{example}

\begin{remark}
\label{remark:analyticnormalization}
The existence of normalizations for analytic varieties is more subtle.  Ordinary
normalizations as in Example \ref{example:ordinarynormalization} do exist.  However,
while this is elementary for algebraic varieties, for analytic varieties this
is a nontrivial theorem of Oka (see \cite[Theorem II.7.12]{Demailly}).  

To define more general normalizations, the Zariski open set $U$ should be replaced
by an open set whose complement is a closed analytic subvariety.
More general normalizations need not exist.  However, it is true that if they exist
they are unique (see \cite[Lemma II.7.11]{Demailly}\footnote{This reference only
proves uniqueness for the ordinary normalization, but its proof is elementary and
applies verbatim for more general normalizations.}).  We will use
this uniqueness in our proof of Theorem \ref{theorem:mumfordriemann} below.
\end{remark}

\subsection{Algebraization}

Let $M$ be an analytic variety.  Our next goal is to understand
when there exists an algebraic variety $X$ with $X^{\an} \cong M$.
This is a subtle question.  Serre's GAGA theorems \cite{SerreGAGA}
show that this always holds in the projective setting, and moreover
$X$ is unique up to algebraic isomorphism.  See Theorem \ref{theorem:gaga} below
for a more precise statement.  This is not true
in general:
\begin{itemize}
\item For $n \geq 2$, a typical complex torus $\C^n/\Lambda$ is
not the analytification of an algebraic variety (see, e.g., \cite[\S I.3]{MumfordAbelianVar}).
See Example \ref{example:nonalgebraic} below for a different kind of example
of an analytic variety that is not the analytification of an algebraic variety.
\item There exist affine analytic varieties $M$ such that
there exist two non-isomorphic affine algebraic varieties
$X$ and $Y$ with $X^{\an} \cong Y^{\an} \cong M$.  In
fact, if $N$ is the analytification of a smooth affine
algebraic curve of genus at least $1$, then there
are uncountably many pairwise non-isomorphic affine algebraic varieties
whose analytifications are isomorphic to $N \times \C$.
See \cite{KimMO}.
\end{itemize}

\subsection{Algebraizing finite ramified covers, I}

What we are really interested in is a relative
version of the above question.
Let $X$ be an algebraic variety, let $M$ be an analytic variety,
and let $f\colon M \rightarrow X^{\an}$ be a finite map.  
Our goal is to describe conditions under which we can 
find an algebraic variety $Y$ equipped with a finite map $g\colon Y \rightarrow X$
such that $Y^{\an}$ is isomorphic to $M$ via an isomorphism fitting
into a commutative diagram
\[\begin{tikzcd}
Y^{\an} \arrow{rr}{\cong} \arrow{dr}{g^{\an}} & & M \arrow{dl}[swap]{f} \\
& X^{\an} &
\end{tikzcd}\]
In this case, we will say that $g\colon Y \rightarrow X$ is an {\em algebraization} of $f\colon M \rightarrow X^{\an}$.
We will say that $f\colon M \rightarrow X^{\an}$ has a {\em unique algebraization} if it has
an algebraization $g\colon Y \rightarrow X$ and if $g'\colon Y' \rightarrow X$ is another
algebraization of $f$ then there exists an isomorphism from $Y$ to $Y'$ fitting into
a commutative diagram
\[\begin{tikzcd}
Y \arrow{rr}{\cong} \arrow{dr}{g} & & Y' \arrow{dl}[swap]{g'} \\
& X &
\end{tikzcd}\]
The easiest case is when $X$ is proper, in which case the existence
of unique algebraizations is a consequence of the proper version
of Serre's GAGA theorems:

\begin{theorem}[{\cite[Expos\'{e} XII, Corollaire 4.6]{SGA1}}]
\label{theorem:gaga}
Let $X$ be a proper algebraic variety, let $M$ be an analytic
variety, and let $f\colon M \rightarrow X^{\an}$ be a finite
map.  Then there exists a unique algebraization $g\colon Y \rightarrow X$
of $f\colon M \rightarrow X^{\an}$.  Moreover, $Y$ is proper, and
if $X$ is projective then $Y$ is projective.
\end{theorem}

\subsection{Algebraizing finite ramified covers, II}

If $X$ is not proper, things are more subtle.  If $M$ and $N$
are analytic varieties, an analytic map $f\colon M \rightarrow N$ is
said to be a {\em finite unbranched cover} if it satisfies the following:
\begin{itemize}
\item All points of $N$ have open neighborhoods $U \subset N$ such that
letting $\tU = f^{-1}(U)$, for some finite indexing set $I \neq \emptyset$ we have an
isomorphism $U \times I \cong \tU$ of analytic spaces fitting into a commutative 
diagram
\[\begin{tikzcd}
U \times I \arrow{rr}{\cong} \arrow{dr}{\pi} & & \tU \arrow{dl}[swap]{f} \\
& U &
\end{tikzcd}\]
\end{itemize}
This implies that $f\colon M \rightarrow N$ is a finite-sheeted covering
space in the sense of algebraic topology, but is stronger than that
if $N$ is singular.\footnote{Indeed, there exist bijective
maps $f\colon M \rightarrow N$ of analytic varieties that are homeomorphisms
but not isomorphisms of analytic varieties.  These are $1$-sheeted covering spaces in the sense
of algebraic topology, but not in the sense we are talking about.}
For finite unbranched covers, we can appeal to the generalized Riemann existence
theorem, which was first proved by Grauert--Remmert \cite{GrauertRemmert}.
See \cite[Expos\'{e} XII, \S 5]{SGA1} for alternate proofs.

\begin{theorem}[{Generalized Riemann existence theorem \cite{GrauertRemmert}}]
\label{theorem:generalizedriemann}
Let $X$ be an algebraic variety, let $M$ be an analytic
variety, and let $f\colon M \rightarrow X^{\an}$ be a finite
unbranched cover.  Then there exists a unique algebraization $g\colon Y \rightarrow X$
of $f\colon M \rightarrow X^{\an}$.
\end{theorem}

However, for general finite maps there might not exist an algebraization:

\begin{example}
\label{example:nonalgebraic}
Let $M = \Set{$(x,y) \in \C^2$}{$y^2 = \sin(x)$}$ and let $f\colon M \rightarrow \C$
be the finite analytic map $f(x,y) = x$.  The map $f\colon M \rightarrow \C$ has
no algebraization since it is branched over infinitely many points.  In fact,
$M$ is homeomorphic to an infinite-genus topological surface, so it
is not the analytification of any algebraic variety.
\end{example}

A version of the generalized Riemann existence theorem that applies in
settings beyond finite unbranched covers is stated on the first page of Mumford's
paper \cite{MumfordAbelian}, where it is attributed to Artin and Serre.  
Since he omits the proof, we include one for completeness:\footnote{The
statement of the theorem in \cite{MumfordAbelian} requires
$X$ to be normal, but our proof does not need this.}

\begin{theorem}[{Artin \& Serre, cf. \cite{MumfordAbelian}}]
\label{theorem:mumfordriemann}
Let $X$ be an irreducible algebraic variety, let $M$ be an irreducible analytic variety, and let
$f\colon M \rightarrow X^{\an}$ be a finite map.  Assume that $M$ is normal
and that there exists a Zariski open set $U \subset X$ such that $V = f^{-1}(U^{\an})$
is connected and dense in $M$ and $f\colon V \rightarrow U^{\an}$ is a finite unbranched cover.
Then there exists a unique algebraization $g\colon Y \rightarrow X$
of $f\colon M \rightarrow X^{\an}$.
\end{theorem}
\begin{proof}
By the ordinary generalized Riemann existence theorem (Theorem \ref{theorem:generalizedriemann}), 
there is a unique algebraization $g\colon \tU \rightarrow U$ of $f\colon V \rightarrow U^{\an}$.
Since $M$ is normal, $V$ is also normal.  Since $V$ is connected and normal, so is $\tU$.
In particular, $\tU$ is irreducible.
Let $g\colon Y \rightarrow X$ be the normalization of $X$ with respect to 
$g\colon \tU \rightarrow U$.  

Since $M$ is normal, both $g^{\an}\colon Y^{\an} \rightarrow X^{\an}$ and $f\colon M \rightarrow X^{\an}$
are normalizations of $X^{\an}$ with respect to $f\colon V \rightarrow U^{\an}$.
By the uniqueness of normalizations of analytic varieties (see Remark \ref{remark:analyticnormalization}), 
there exists an isomorphism from $Y^{\an}$ to $M$ fitting into a commutative diagram
\[\begin{tikzcd}
Y^{\an} \arrow{rr}{\cong} \arrow{dr}{g^{\an}} & & M \arrow{dl}[swap]{f} \\
& X^{\an} &
\end{tikzcd}\]
In other words, $g\colon Y \rightarrow X$ is an algebraization
of $f\colon M \rightarrow X^{\an}$.  

To see that $g$ is unique, let $g'\colon Y' \rightarrow X$ be another algebraization
of $f\colon M \rightarrow X^{\an}$.  Set $\tU' = (g')^{-1}(U)$, so
$g'\colon \tU' \rightarrow U$ is an algebraization of the finite unbranched
cover $f\colon V \rightarrow U^{\an}$.  The uniqueness part of
the ordinary generalized Riemann existence theorem (Theorem \ref{theorem:generalizedriemann})
implies that $g'\colon \tU' \rightarrow U$ is isomorphic to $g\colon \tU \rightarrow U$.
Identifying $\tU' \subset Y'$ with $\tU \subset Y$ via such an isomorphism, we see
that $g\colon Y \rightarrow X$ and $g'\colon Y' \rightarrow X$ are both
normalizations of $X$ with respect to the finite map $\tU \rightarrow U$.
By the uniqueness of normalizations, we conclude that $g\colon Y \rightarrow X$
is isomorphic to $g'\colon Y' \rightarrow X$.
\end{proof}

\section{The moduli space of curves}
\label{section:modulicurves}

We now turn to the moduli space of curves.
Good general references for the algebro-geometric material in this
section are \cite{ACGH, HarrisMorrison}.  See \cite{FarbMargalitPrimer, HubbardTeich, IvanovSurvey}
for more information on Teichm\"{u}ller space and the mapping class group.

\subsection{Technical framework}

There are many different senses in which the moduli space of curves is a ``space''.
The most sophisticated point of view is that it is an algebraic stack; however, that
language is unnecessary for our purposes.  We prefer to simply treat
it as a complex algebraic variety.  Therefore,
throughout this paper we work with coarse moduli spaces and their associated
analytic spaces, rather than with moduli stacks.

\subsection{Setup}

Fix some $g,n \geq 0$ such that $\Sigma_{g,n}$ is hyperbolic, i.e., such
that $2g-2+n > 0$.  Recall that under our conventions the mapping class group $\Mod_{g,n}$
can permute the punctures.  Let $\PMod_{g,n}$ be the subgroup of $\Mod_{g,n}$ consisting
of mapping classes that fix the punctures pointwise.  This is often called the pure mapping class group.

\begin{remark}
We stated Theorem \ref{maintheorem:boggi} using $\Mod_{g,n}$ since that gives
a more general theorem.  However, most of the literature on the moduli space of curves 
uses $\PMod_{g,n}$.  Because of this, the first step of the proof of Theorem \ref{maintheorem:boggi} 
reduces it to $\PMod_{g,n}$.
\end{remark}

\subsection{Teichm\"{u}ller space}

Let $\Hyp(\Sigma_{g,n})$ be the set of finite-volume complete hyperbolic metrics on $\Sigma_{g,n}$.
The group $\Diff^{+}(\Sigma_{g,n})$ acts on $\Hyp(\Sigma_{g,n})$.  Letting $\Diff_0(\Sigma_{g,n})$ be
the group of diffeomorphisms of $\Sigma_{g,n}$ that are isotopic to the identity,
Teichm\"{u}ller space is the space
\[\Teich_{g,n} = \Hyp(\Sigma_{g,n}) / \Diff_0(\Sigma_{g,n})\]
of isotopy classes of finite-volume complete hyperbolic
metrics on $\Sigma_{g,n}$.  The space $\Teich_{g,n}$ is a complex manifold
of complex dimension $3g+n-3$.  It is contractible; indeed, it is homeomorphic to
$\R^{6g+2n-6}$.  The mapping class group $\PMod_{g,n}$ acts properly discontinuously 
on $\Teich_{g,n}$ by complex analytic automorphisms.  The quotient
$\Teich_{g,n} / \PMod_{g,n}$ is therefore a complex analytic space.

\subsection{The moduli space of curves}

Let $\Moduli_{g,n}$ be the moduli space of smooth connected projective genus-$g$ algebraic curves with
$n$ distinct ordered marked points.  This is an algebraic variety with
\[\Moduli_{g,n}^{\an} \cong \Teich_{g,n}/\PMod_{g,n}.\]
To simplify our notation, we will not distinguish
between $\Moduli_{g,n}$ and $\Moduli_{g,n}^{\an}$ unless it is necessary
for clarity.  The points of $\Moduli_{g,n}$ are isomorphism classes of
tuples $\cS = (S;p_1,\ldots,p_n)$, where $S$ is a smooth projective genus-$g$ algebraic curve
and $p_1,\ldots,p_n \in S$ are distinct points.  We will often slightly abuse notation
and call $\cS$ a point of $\Moduli_{g,n}$.  If $n=0$, then we will just call
the curve $S$ a point of $\Moduli_{g,n}$.

The group $\PMod_{g,n}$ need not act freely on $\Teich_{g,n}$, so $\Moduli_{g,n}$ is not
necessarily smooth.\footnote{It is not the case that points of $\Teich_{g,n}$ with nontrivial stabilizers
always project to singular points of $\Moduli_{g,n}$.  We will study the singular points of
$\Moduli_{g,n}$ in \S \ref{section:curvesauto}.}  However, we will later see that it is normal (see Lemma \ref{lemma:modulinormal}).
The algebraic variety $\Moduli_{g,n}$ is also quasiprojective, and we will
describe a projective compactification of it in \S \ref{section:delignemumford} below.

\subsection{Low-complexity examples}

In some low-complexity examples, we can describe the algebraic variety $\Moduli_{g,n}$ completely:

\begin{example}
\label{example:moduli0n}
Fix $n \geq 3$.  The points of $\Moduli_{0,n}$ are isomorphism classes of
$n$ distinct ordered points on $\bbP^1 = \C \cup \{\infty\}$.  The automorphism group of
$\bbP^1$ is $\PGL_2(\C)$.  This group acts $3$-transitively on $\bbP^1$,
and the stabilizer of any three distinct points is trivial.  Each isomorphism
class of $n$ ordered distinct points on $\bbP^1$ therefore has a unique
representative where the first three points are $0$ and $1$ and $\infty$.
From this, we see that
\[\Moduli_{0,n} \cong \Set{$(z_1,\ldots,z_{n-3}) \in (\C \setminus \{0,1\})^{n-3}$}{$z_i \neq z_j$ for all $1 \leq i < j \leq n-3$}.\]
This is the Zariski open subset of $\C^{n-3}$ whose complement is the vanishing
locus of the polynomial
\[f(z_1,\ldots,z_{n-3}) = \left(\prod_{1 \leq i \leq n-3} z_i(z_i-1)\right) \left(\prod_{1 \leq i < j \leq n-3} (z_i-z_j)\right).\]
In particular, $\Moduli_{0,n}$ is a smooth affine variety.\footnote{More generally, if $g \in \C[z_1,\ldots,z_d]$
is a polynomial, then the complement of the vanishing locus of $g$ in $\C^d$ is isomorphic to the affine subvariety
of $\C^d \times \C$ defined by the vanishing of the polynomial $h(z_1,\ldots,z_d,w) = w g(z_1,\ldots,z_d)-1$.}
\end{example}

\begin{example}
\label{example:moduli2}
Every smooth connected projective genus-$2$ algebraic curve $S$ is hyperelliptic,\footnote{We will
describe general hyperelliptic curves later in \S \ref{section:hyperelliptic}.} that is, it
has a degree-$2$ map $f\colon S \rightarrow \bbP^1$.  This map $f$ is unique up to postcomposition
by automorphisms of $\bbP^1$.  By the Riemann--Hurwitz formula, $f$ has $6$ branch points.  Up
to isomorphism, $S$ is determined by the images of those branch points in $\bbP^1$.  

It follows that $\Moduli_2$ is isomorphic to the moduli space of $6$ unordered distinct points
in $\bbP^1$.  Letting $S_6$ denote the symmetric group on $6$ points, the group $S_6$ acts
on $\Moduli_{0,6}$ by permuting the points and the above shows that $\Moduli_2 \cong \Moduli_{0,6}/S_6$.
The quotient of an affine variety by a finite group is affine, so we see in particular
that $\Moduli_2$ is an affine variety.
\end{example}

\begin{example}
\label{example:moduli11}
Points of $\Moduli_{1,1}$ are pairs $E = (S;p_1)$, where $S$ is a smooth connected projective genus-$1$
curve and $p_1 \in S$.  These are also called {\em elliptic curves}.  They have a canonical abelian group structure
with $p_1 \in S$ the identity.  From this, we get a canonical order-$2$ inversion
automorphism $i \in \Aut(E)$.  The map $i\colon S \rightarrow S$ fixes four points (including $p_1$), and the quotient
$S/i$ is $\bbP^1$.  

The isomorphism class of the elliptic curve $E$ is determined by the image
in $\bbP^1$ of these four fixed points.  One of these is the image of $p_1$,
and the other three are unordered.  It follows that $\Moduli_{1,1} \cong \Moduli_{0,4} /S_3$, where
$S_3$ acts by permuting the last three marked points.  In particular, $\Moduli_{1,1}$ is an
affine variety.
\end{example}

It is harder to describe a general $\Moduli_{g,n}$, which
in particular is typically far from an affine variety.  However, 
as we will describe next we can construct algebraic maps into $\Moduli_{g,n}$ by constructing
families of genus-$g$ curves with $n$ marked points.

\subsection{Classifying maps for families}
\label{section:moduliuniversal}

In an appropriate sense, the variety $\Moduli_{g,n}$ classifies families of curves with marked points.
One version of this is as follows.
Let $\pi\colon E \rightarrow B$ be a map of varieties and let $\sigma_1,\ldots,\sigma_n\colon B \rightarrow E$
be sections of $\pi$ such that the following hold:
\begin{itemize}
\item For some $d \geq 0$, the variety $B$ is a smooth $d$-dimensional variety and
the variety $E$ is a smooth $(d+1)$-dimensional variety.
\item The analytification of $\pi\colon E \rightarrow B$ is a proper holomorphic submersion.\footnote{If we allowed
$E$ and $B$ to be singular, then we would instead need to impose the more algebraic condition that
$\pi$ is flat and proper.}
\item For all $b \in B$, the preimage $\pi^{-1}(b)$ is a smooth connected projective genus-$g$ algebraic
curve.
\item The images of $\sigma_1,\ldots,\sigma_n$ are disjoint.
\end{itemize}
For all $b \in B$, we have a point
\[f(b) = (\pi^{-1}(b);\sigma_1(b),\ldots,\sigma_n(b)) \in \Moduli_{g,n}.\]
The resulting map $f\colon B \rightarrow \Moduli_{g,n}$ is then algebraic.

\begin{remark}
It is not necessarily true that every map to $\Moduli_{g,n}$ comes from the above construction.
This can be fixed by working in the category of {\em algebraic stacks}.
For geometric topologists who are interested in this, we recommend the beautiful
and elementary survey \cite{BehrendStacks}, which starts by describing
the moduli stack of Euclidean triangles.
\end{remark}

\section{Locus of curves with automorphisms}
\label{section:curvesauto}

Our next main goal is to study finite branched covers of the moduli
space of curves.  This requires a preliminary study of curves with
automorphisms.  Fix $g,n \geq 0$ such that $\Sigma_{g,n}$ is hyperbolic.

\subsection{Goal of section}

Consider a point $\cS = (S;p_1,\ldots,p_n)$ of $\Moduli_{g,n}$.  Its
automorphism group $\Aut(\cS)$ is the group of algebraic isomorphisms
$f\colon S \rightarrow S$ such that $f(p_i) = p_i$ for all $1 \leq i \leq n$.
As we will describe below, in most cases there are points
$\cS \in \Moduli_{g,n}$ such that $\Aut(\cS) = 1$.  The two exceptions
are $(g,n) = (2,0)$ and $(g,n) = (1,1)$, where all $\cS \in \Moduli_{g,n}$
have $|\Aut(\cS)| \geq 2$.  
Define $\Moduli_{g,n}^{\sym}$ to be the subspace of points $\cS \in \Moduli_{g,n}$
such that:
\begin{itemize}
\item $|\Aut(\cS)| > 1$ if $(g,n) \notin \{(2,0),(1,1)\}$; and
\item $|\Aut(\cS)| > 2$ if $(g,n) \in \{(2,0),(1,1)\}$.
\end{itemize}
Our main goal in this section is to prove that this is a proper Zariski
closed subset of $\Moduli_{g,n}$.  In most cases, we will do this by
proving that $\Moduli_{g,n}^{\sym}$ is the singular locus of $\Moduli_{g,n}$.
However, in a few cases this is not true, so we will need to work a little
harder.

\begin{remark}
We have $\Moduli_{0,n}^{\sym} = \emptyset$; indeed, since we have assumed
that $\Sigma_{0,n}$ is hyperbolic we must have $n \geq 3$, and an automorphism
of $\bbP^1$ fixing $n \geq 3$ points must be the identity.  We will therefore
focus on the case $g \geq 1$.
\end{remark}

\subsection{Hyperelliptic curves}
\label{section:hyperelliptic}

We start by studying an important special class of curves with automorphisms.
Assume that $g \geq 1$.
A smooth connected projective genus-$g$ algebraic curve $S$ is {\em hyperelliptic} if
there exists a degree-$2$ algebraic map $f\colon S \rightarrow \bbP^1$.  Such an
$f$ is branched at $2g+2$ points.  For $g \geq 2$ it is unique up to postcomposition
with automorphisms of $\bbP^1$ if it exists.  For $g = 1$ it is not unique in this way, but it is unique up
to postcomposition with automorphisms of $\bbP^1$ if we specify
one of the branch points.  

The automorphisms of $\bbP^1 = \C \cup \{\infty\}$ act $3$-transitively
on $\bbP^1$, so after composing $f$ with an automorphism of $\bbP^1$ we can assume
that three of the branch values are $\{0,1,\infty\}$.  Let
$\lambda_1,\ldots,\lambda_{2g-1} \in \C \setminus \{0,1\}$ be the other branch values.  

The variety $\bbP^1$ can be covered by two affine coordinate charts.  Let $x$ be the coordinate
centered at $0 \in \bbP^1$ and let $u$ be the coordinate centered at $\infty$,
so for $x \neq 0$ we have $u = 1/x$.  Over these two coordinate charts,
$S$ is isomorphic to the affine varieties defined by the following equations:\footnote{For $g \geq 2$, be warned
that $S$ is {\em not} the closure of either of these affine varieties in $\bbP^2$.  In fact,
the closures of both of these affine varieties are singular.}
\begin{align*}
y^2 &= x(x-1)(x-\lambda_1) \cdots (x-\lambda_{2g-1}),\\
v^2 &= u(1-u)(1-\lambda_1 u) \cdots (1-\lambda_{2g-1} u).
\end{align*}
To see how these two affine varieties are glued together to form $S$, plug
$x = 1/u$ into the first equation and clear denominators to get
\[(y u^{g+1})^2 = u(1-u)(1-\lambda_1 u) \cdots (1-\lambda_{2g-1} u).\]
It follows that a point $(x,y)$ in the first affine variety with $x \neq 0$
is identified with the point $(u,v) = (1/x,y/x^{g+1})$ of the second affine variety.

In this model, the involution $i\colon S \rightarrow S$ that multiplies $y$ and $v$ by
$-1$ is called the {\em hyperelliptic involution} 
associated to $f$.  Identifying $S$ with $\Sigma_g$, this
hyperelliptic involution $i$ is smoothly conjugate to the following involution
$\iota\colon \Sigma_g \rightarrow \Sigma_g$:
\Figure{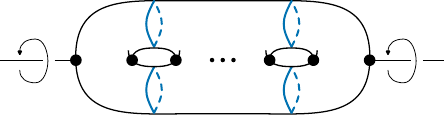}
The map $\iota$ exchanges the $g$ circles on the top and bottom while fixing the
indicated $2g+2$ marked points.  For $0 \leq n \leq 2g+2$, we can puncture
the surface at the first $n$ marked points (read left to right)
and identify $\iota$ with an element of $\PMod_{g,n}$.
Elements of $\PMod_{g,n}$ that are conjugate to $\iota$ are also called hyperelliptic
involutions.

\subsection{Hyperelliptic locus}

The {\em hyperelliptic locus}, denoted $\cH_{g,n}$, is
the subspace of $\Moduli_{g,n}$ consisting of points $\cS = (S;p_1,\ldots,p_n)$
such that the genus-$g$ curve $S$ admits a degree-$2$ algebraic map $f\colon S \rightarrow \bbP^1$
that is ramified at each $p_i$.  The resulting hyperelliptic involution
$\iota\colon S \rightarrow S$ fixes each $p_i$.  We emphasize that this
is stronger than merely requiring $S$ to be hyperelliptic.
The preimage of $\cH_{g,n}$ in $\Teich_{g,n}$ is the subspace of points that are fixed
by some hyperelliptic involution in $\PMod_{g,n}$.

The proof of the following lemma is a nice illustration of the interplay between
analytic and algebraic geometry.  It also illustrates how to use the
mapping property from \S \ref{section:moduliuniversal} to prove
nontrivial results without knowing the precise details of the structure
of the algebraic variety $\Moduli_{g,n}$:

\begin{lemma}
\label{lemma:hyperclosed}
Fix $g \geq 1$ and $n \geq 0$ such that $\Sigma_{g,n}$ is hyperbolic.
Then $\cH_{g,n}$ is a Zariski closed subspace of $\Moduli_{g,n}$
of dimension at most $2g-1$.
\end{lemma}

\begin{remark}
The subspace $\cH_{g,n}$ is nonempty if and only if $0 \leq n \leq 2g+2$.  In that
case, with more work one can show that $\cH_{g,n} \cong \Moduli_{0,2g+2}/S_{2g+2-n}$, where
$S_{2g+2-n}$ fixes the first $n$ marked points.  This
generalizes Example \ref{example:moduli2}, and implies in particular that
$\cH_{g,n}$ has dimension exactly equal to $2g-1$.
\end{remark}

\begin{proof}[Proof of Lemma \ref{lemma:hyperclosed}]
The subspace $\cH_{g,n}$ is empty if $n > 2g+2$, so we can assume
without loss of generality that $0 \leq n \leq 2g+2$.  Below we will prove two things:
\begin{itemize}
\item[(a)] The subset $\cH_{g,n}$ of $\Moduli_{g,n}$ is closed in the
analytic topology.
\item[(b)] There is an algebraic map $f\colon \Moduli_{0,2g+2} \rightarrow \Moduli_{g,n}$ whose
image is $\cH_{g,n}$.
\end{itemize}
Before doing this, we explain why these facts imply the lemma.

Since $\cH_{g,n}$ is the image of an algebraic map by (b),
a theorem of Chevalley (see, e.g., \cite[\S I.8, Corollary 2]{MumfordRed})
says that it is constructible, that is, a finite union of sets
of the form $U \cap C$ with $U$ Zariski open and $C$ Zariski closed.  A standard
result in algebraic geometry (see, e.g., \cite[Expos\'e XII, Corollary 2.3]{SGA1} or \cite[Theorem 5.4.4]{ArtinBook}) says that
a constructible subspace of an algebraic variety that is closed in
the analytic topology is closed in the Zariski topology, so by (a)
we deduce that $\cH_{g,n}$ is Zariski closed.  Finally, since $\Moduli_{0,2g+2}$
has dimension $2g-1$ it follows from (b) that $\cH_{g,n}$ has dimension
at most $2g-1$.

We divide our proof of (a) and (b) into the following two claims:

\begin{claim}{1}
\label{claim:hyperclosed.1}
The hyperelliptic locus $\cH_{g,n}$ is a closed subspace of $\Moduli_{g,n}$ in its analytic topology.
\end{claim}

It is enough to prove that the preimage of $\cH_{g,n}$ in the Teichm\"{u}ller space
$\Teich_{g,n}$ is closed.  This preimage is the set
\[C = \bigcup_{\substack{\iota \in \PMod_{g,n} \\ \text{hyperelliptic involution}}} \Teich_{g,n}^{\iota},\]
where $\Teich_{g,n}^{\iota}$ is the fixed locus of $\iota$ in $\Teich_{g,n}$.  To prove that $C$
is closed, consider a limit point $x \in \Teich_{g,n}$ of $C$.  We must prove that
$x \in C$.  Let $(\PMod_{g,n})_x$ be the stabilizer of $x$.
Since $\PMod_{g,n}$ acts properly discontinuously on $\Teich_{g,n}$, there exists
an open neighborhood $B$ of $x$ such that the only $f \in \PMod_{g,n}$ with
$f(B) \cap B \neq \emptyset$ are those $f$ that lie in $(\PMod_{g,n})_x$.  It follows that
\[C \cap B = \bigcup_{\substack{\iota \in (\PMod_{g,n})_x \\ \text{hyperelliptic involution}}} \Teich_{g,n}^{\iota} \cap B.\]
Since $(\PMod_{g,n})_x$ is a finite group, this is a finite union.  Since each $\Teich_{g,n}^{\iota}$
is closed, we deduce that $C \cap B$ is a closed subset of $B$, and hence that $x \in C$, as desired.

\begin{claim}{2}
\label{claim:hyperclosed.2}
There is an algebraic map $f\colon \Moduli_{0,2g+2} \rightarrow \Moduli_{g,n}$ whose
image is $\cH_{g,n}$.
\end{claim}

Set
\[B = \Set{$(z_1,\ldots,z_{2g-1}) \in (\C \setminus \{0,1\})^{2g-1}$}{$z_i \neq z_j$ for all $1 \leq i < j \leq 2g-1$},\]
so $B \cong \Moduli_{0,2g+2}$ (see Example \ref{example:moduli0n}).  We will define a
family of smooth projective genus-$g$ curves $\pi\colon E \rightarrow B$ with sections
$\sigma_1,\ldots,\sigma_n\colon B \rightarrow E$
in the sense of \S \ref{section:moduliuniversal}.  For $b \in B$, the fiber $\pi^{-1}(b)$
will be the hyperelliptic curve whose branch image consists of the unordered set of $2g+2$ points
on $\bbP^1$ obtained by forgetting the ordering on the points corresponding to $b$.  The
sections mark fixed points of the hyperelliptic involution on $\pi^{-1}(b)$.
Via the construction in \S \ref{section:moduliuniversal},
we get an algebraic map $f\colon B \rightarrow \Moduli_{g,n}$ whose image is $\cH_{g,n}$.  
 
It remains to construct $\pi\colon E \rightarrow B$ and the sections
$\sigma_1,\ldots,\sigma_n\colon B \rightarrow E$.  Define
$U \subset B \times \C^2$ and $V \subset B \times \C^2$ to be
\begin{align*}
U &= \Set{$((z_1,\ldots,z_{2g-1}),(x,y)) \in B \times \C^2$}{$y^2 = x(x-1)(x-z_1)\cdots(x-z_{2g-1})$},\\
V &= \Set{$((z_1,\ldots,z_{2g-1}),(u,v)) \in B \times \C^2$}{$v^2 = u(1-u)(1-z_1 u) \cdots (1-z_{2g-1} u)$}.
\end{align*}
Both $U$ and $V$ are smooth affine varieties.  Let $E$ be the result of gluing $U$ to $V$ as follows:
\begin{itemize}
\item For $\bz = (z_1,\ldots,z_{2g-1}) \in B$ and $(\bz,(x,y)) \in U$ with $x \neq 0$, identify
$(\bz,(x,y)) \in U$ with $(\bz,(1/x,y/x^{g+1})) \in V$.
\end{itemize}
This is a smooth algebraic variety covered by the open affine subsets $U$ and $V$, and
the projections $U \rightarrow B$ and $V \rightarrow B$ onto the first coordinate
give an algebraic map $\pi\colon E \rightarrow B$.

By the discussion in \S \ref{section:hyperelliptic}, for $\bz \in B$ the fiber $\pi^{-1}(\bz)$ is indeed the
hyperelliptic curve corresponding to $\bz$.  For $1 \leq i \leq n$, define a section
$\sigma_i\colon B \rightarrow E$ as follows.  Consider $\bz = (z_1,\ldots,z_{2g-1}) \in B$.
\begin{itemize}
\item If $1 \leq i \leq 2g-1$, then $\sigma_i(\bz)$ is the image in $E$ of the point
$(\bz,(z_i,0)) \in U$.
\item If $i = 2g$, then $\sigma_i(\bz)$ is the image in $E$ of the point $(\bz,(0,0)) \in U$.
\item If $i = 2g+1$, then $\sigma_i(\bz)$ is the image in $E$ of the point $(\bz,(1,0)) \in U$.
\item If $i = 2g+2$, then $\sigma_i(\bz)$ is the image in $E$ of the point $(\bz,(0,0)) \in V$.
\end{itemize}
Since we have assumed that $0 \leq n \leq 2g+2$, these account for all
possible values of $i$.  We remark that we marked $\{0,1,\infty\}$ last rather than
first to avoid index shifts in our notation.

We must check that the analytification of $\pi$
is a proper submersion.  That it is a submersion is immediate from the equations defining $U$ and $V$,
so all that remains to check is that $\pi$ is proper.  For this, note that
the projections 
\begin{alignat*}{2}
U &\rightarrow B \times \C &\quad\quad V &\rightarrow B \times C \\
(\bz,(x,y)) & \mapsto (\bz,x) &\quad\quad  (\bz,(u,v)) &\mapsto (\bz,u)
\end{alignat*}
glue to give a finite degree-$2$ map $E \rightarrow B \times \bbP^1$, and $\pi$ factors as
$E \rightarrow B \times \bbP^1 \rightarrow B$.
Since both of these maps are proper, it follows that $\pi$ is proper.
\end{proof}

\subsection{Linear quotients of affine space}

Before continuing our study of curves with automorphisms, we pause
to prove the following standard lemma:

\begin{lemma}[{cf.\ Remark \ref{remark:smoothcriterion}}]
\label{lemma:notmanifold}
Let $W$ be a finite-dimensional complex vector space and let $F < \GL(W)$ be
a nontrivial finite group.  Let $W/F$ be the complex algebraic variety obtained
by quotienting $W$ by $F$ and let $\pi\colon W \rightarrow W/F$ be the projection.
For all $f \in F$ with $f \neq 1$, assume that
\[W^f = \Set{$w \in W$}{$f(w) = w$}\]
has codimension at least $2$.  Then $W/F$ is not smooth at $\pi(0)$.
\end{lemma}
\begin{proof}
All topological statements we make in this proof refer to the analytic topology.
Let $n = \dim(W)$.  Our assumption on the fixed loci of $F$ implies that $n \geq 2$.
The variety $W/F$ is also $n$-dimensional.
Assume for the sake of contradiction that $W/F$ is smooth at $\pi(0)$.
This implies the following:

\begin{unnumberedclaim}
The variety $W/F$ is smooth and contractible.
\end{unnumberedclaim}
\begin{proof}[Proof of claim]
For $0 \leq r \leq 1$, the linear map $m_r\colon W \rightarrow W$ defined by $m_r(w) = r w$ descends
to a map $\om_r\colon W/F \rightarrow W/F$.  We have $\om_1 = \text{id}$,
and $\om_0$ is a constant map.  It follows that $W/F$ is contractible.  Next,
since $W/F$ is smooth at $\pi(0)$, there exists a Zariski open neighborhood $U$ 
of $\pi(0)$ consisting of smooth points.  For a point $x \in W/F$,
we can find some $0 < r \leq 1$ such that $\om_r(x) \in U$, so $\om_r(x)$ is a smooth
point.  Since
$0 < r \leq 1$, the map $\om_r$ is an isomorphism.  We conclude that $x$ is a smooth
point, as desired.
\end{proof}

Set
\[Z = \bigcup_{\substack{f \in F \\ f \neq 1}} W^f.\]
This is a subvariety of codimension at least $2$ in $W$.  Since
$\pi\colon W \rightarrow W/F$ is a finite map, it follows that $\oZ = \pi(Z)$ is also
a subvariety of codimension at least $2$ in $W/F$ (see \S \ref{section:finitemaps}).  General
position arguments show that if $X$ is a smooth complex variety and $Y \subset X$
is a subvariety of (complex) codimension at least $2$, then
the inclusion $X-Y \hookrightarrow X$ induces an isomorphism on $\pi_1$
for any choice of basepoint.  Letting $w_0 \in W-Z$ be a basepoint and
$\ow_0 = \pi(w_0)$, it follows that
\begin{align*}
\pi_1(W-Z,w_0) \cong &\pi_1(W,w_0) = 1, \\
\pi_1(W/F-\oZ,\ow_0) \cong & \pi_1(W/F,\ow_0) = 1.
\end{align*}
The group $F$ acts freely on $W-Z$ with $(W -Z)/F \cong W/F - \oZ$,
so by covering space theory we have $\pi_1(W/F - \oZ,\ow_0) \cong F \neq 1$, a contradiction.
\end{proof}

\begin{remark}
\label{remark:notmanifold}
In fact, the Chevalley--Shephard--Todd theorem (see, e.g., \cite[\S 3.5]{LehrerTaylor}) implies that $W/F$ is smooth if and only
if $F$ is generated by {\em pseudo-reflections}, i.e., linear maps whose fixed set
has codimension $1$.  See \cite{Prill} for a complete classification of complex
varieties of the form $W/F$.  
\end{remark}

\subsection{Locus of curves with automorphisms, high complexity}

Assume that $(g,n) \notin \{(2,0),(1,1)\}$, or equivalently $g+n \geq 3$.  Recall that
$\Moduli_{g,n}^{\sym}$ is the subspace of $\Moduli_{g,n}$ consisting of
points $\cS$ with $\Aut(\cS)$ nontrivial.  The following theorem
was first proved by Rauch \cite{RauchSingular} when $n = 0$.  We do not
know a reference proving the general case, but its proof follows
the outline of the proof in \cite{RauchSingular}.

\begin{theorem}[{Rauch \cite{RauchSingular}}]
\label{theorem:rauchsym}
Fix $g \geq 1$ and $n \geq 0$ such that $g+n \geq 3$.  The following hold:
\begin{itemize}
\item[(i)] If $g+n \geq 4$, then $\Moduli_{g,n}^{\sym}$ is the singular locus of $\Moduli_{g,n}$.
\item[(ii)] If $g+n = 3$, then $\Moduli_{g,n}^{\sym}$ is the union of $\cH_{g,n}$ and the singular locus
of $\Moduli_{g,n}$.
\end{itemize}
\end{theorem}

\begin{remark}
As was noted by Rauch \cite{RauchSingular} in the closed case, a slightly more careful analysis
shows that if $g+n = 3$ and $\cS$ is a point of $\cH_{g,n}$ whose automorphism group
is generated by the hyperelliptic involution, then $\cS$ is a smooth point of $\Moduli_{g,n}$.
\end{remark}

Before proving Theorem \ref{theorem:rauchsym}, we point out the following corollary:

\begin{corollary}
\label{corollary:symclosed}
Fix $g \geq 1$ and $n \geq 0$ such that $\Sigma_{g,n}$ is hyperbolic.  Assume
that $g+n \geq 3$.  Then $\Moduli_{g,n}^{\sym}$ is a proper Zariski closed subvariety of $\Moduli_{g,n}$.
\end{corollary}
\begin{proof}
When $g+n \geq 4$ this is immediate from Theorem \ref{theorem:rauchsym} together with
the fact that the singular locus of a variety is always a proper Zariski closed subvariety.
When $g+n = 3$, Theorem \ref{theorem:rauchsym}
along with the fact that $\cH_{g,n}$ is a Zariski closed subspace of $\Moduli_{g,n}$ (Lemma \ref{lemma:hyperclosed})
implies that $\Moduli_{g,n}^{\sym}$ is a Zariski closed subvariety of $\Moduli_{g,n}$.  To see that
it is a proper subvariety, observe the following:
\begin{itemize}
\item The dimension of each irreducible component of the singular locus of a variety $X$ is strictly
smaller than $\dim(X)$.  In particular, the dimension of each component of the singular locus of
$\Moduli_{g,n}$ is strictly less than $\dim(\Moduli_{g,n}) = 3g-3+n = 2g$.
\item By Lemma \ref{lemma:hyperclosed}, the dimension of $\cH_{g,n}$ is at most $2g-1 < \dim(\Moduli_{g,n})$.\qedhere
\end{itemize}
\end{proof}

\begin{proof}[Proof of Theorem \ref{theorem:rauchsym}]
Consider a point $x$ of Teichm\"{u}ller space $\Teich_{g,n}$.  Letting
$W = T_x \Teich_{g,n}$ be the tangent space to $\Teich_{g,n}$ at $x$, the finite stabilizer
group $F = (\PMod_{g,n})_x$ acts linearly on $T_x \Teich_{g,n}$.  There exists an $F$-invariant
open neighborhood $U$ of the origin in $W$ such that the image of $x$ in $\Moduli_{g,n}$ has
an analytic neighborhood isomorphic to $U/F$.  It follows that the image of $x$ in
$\Moduli_{g,n}$ is a singular point if and only if the origin in $W$ maps to a singular point
of $W/F$.  We will check this in the cases claimed by the theorem 
using Lemma \ref{lemma:notmanifold}.  This requires checking that the fixed subspace
of each nontrivial $f \in F$ has codimension at least $2$.

The codimensions of fixed subspaces are unchanged by dualizing, so it is enough
to verify the hypotheses of Lemma \ref{lemma:notmanifold} for the action of
$F$ on the cotangent space $W^{\ast}$.
The point $x \in \Teich_{g,n}$ corresponds to a point $\cS = (S;p_1,\ldots,p_n)$ of
$\Moduli_{g,n}$ together with an isotopy class of orientation-preserving
diffeomorphism $\Sigma_{g,n} \cong S \setminus \{p_1,\ldots,p_n\}$ taking the
$i^{\text{th}}$ puncture of $\Sigma_{g,n}$ to $p_i$.  We have $F = \Aut(\cS)$, 
and the cotangent space to $\Teich_{g,n}$ at $x$ is the space
$Q(S;p_1,\ldots,p_n)$ of meromorphic quadratic differentials on $S$ 
having at most a simple pole at each $p_i$ and no other poles
(see \cite[Proposition 6.6.2]{HubbardTeich}).
In light of Lemma \ref{lemma:notmanifold}, it is therefore enough to prove Lemma \ref{lemma:invariantdifferentials}
below.
\end{proof}

The above proof used:

\begin{lemma}
\label{lemma:invariantdifferentials}
Fix $g \geq 1$ and $n \geq 0$ such that $g+n \geq 3$.  Let $\cS = (S;p_1,\ldots,p_n)$
be a point of $\Moduli_{g,n}$ and let $f\colon S \rightarrow S$ be a nontrivial
automorphism of $\cS$.  If $g+n = 3$, then assume that $f$ is not a hyperelliptic
involution.  Then the subspace
\[Q(S;p_1,\ldots,p_n)^f = \Set{$q \in Q(S;p_1,\ldots,p_n)$}{$f^{\ast}(q) = q$}\]
has codimension at least $2$ in $Q(S;p_1,\ldots,p_n)$.
\end{lemma}
\begin{proof}
By Riemann--Roch, the vector space $Q(S;p_1,\ldots,p_n)$ is $(3g-3+n)$-dimensional.  We must therefore
prove that $Q(S;p_1,\ldots,p_n)^f$ is at most $(3g-5+n)$-dimensional.  Let $T$ be the quotient of $S$
by $f$ and let $\pi\colon S \rightarrow T$ be the projection.  Letting $d \geq 2$ be the order
of $f$, the map $\pi$ is a degree-$d$ branched cover.  The branch points of $\pi$ in $S$
must include $p_1,\ldots,p_n$.  Let $b_1,\ldots,b_r \in T$
be the branch values.  For $1 \leq i \leq r$, the map $\pi$ has the same ramification
degree $e_i \geq 2$ at each point of $\pi^{-1}(b_i)$.  This $e_i$ divides $d$, and
there are $d/e_i$ points in the fiber $\pi^{-1}(b_i)$ each of which is fixed by $f^{d/e_i}$.
Finally, let $h \geq 0$ be the genus of $T$.  The Riemann--Hurwitz formula therefore
says that
\[2g-2 = d\left(2h-2 + \sum_{i=1}^r\left(1-\frac{1}{e_i}\right)\right).\]
This formula implies that $T \setminus \{b_1,\ldots,b_r\} \cong \Sigma_{h,r}$ is hyperbolic.
Indeed:
\begin{itemize}
\item If $h=0$, then it implies that $\sum_{i=1}^r(1-1/e_i) \geq 2$ and thus
that $r \geq 3$.
\item If $h=1$ and $r=0$, then it implies that $g=1$.  It is impossible in this case for $r=0$; indeed,
our assumption $g+n \geq 3$ implies that $n \geq 2$, and $\pi$ must be ramified at each of the $n$ marked points
$p_1,\ldots,p_n$.
\end{itemize}
Having introduced all this notation, we divide the proof into several steps:

\begin{step}{1}
\label{step.invariantdifferentials.1}
We prove that $Q(S;p_1,\ldots,p_n)^f$ has dimension $3h-3+r$.
\end{step}

Each element of $Q(S;p_1,\ldots,p_n)^f$ is $\pi^{\ast}(q)$ for some meromorphic quadratic
differential $q$ on $T$.  For a meromorphic quadratic differential $q$
on $T$, we will prove that $\pi^{\ast}(q) \in Q(S;p_1,\ldots,p_n)$ if and only
if $q$ has at most a simple pole at each point $b_1,\ldots,b_r$ and
no other poles.  The above will imply that $Q(S;p_1,\ldots,p_n)^f$ is isomorphic
to the vector space of meromorphic quadratic differentials on $T$ with at most
a simple pole at each $b_i$ and no other poles.  
Riemann--Roch says that this has dimension $3h-3+r$, so this will prove the step.

Since $\pi$ must take each $p_i$ to some $b_j$, 
for a meromorphic quadratic differential $q$ on $T$ such that $\pi^{\ast}(q) \in Q(S;p_1,\ldots,p_n)$
it is clear that $q$ cannot have any poles except possibly at the $b_i$.  We now
examine what must hold at a branch value $b_i$.  Consider some $\tb_i \in \pi^{-1}(b_i)$, so
$\pi$ has ramification index $e_i$ at $\tb_i$.  We can choose a holomorphic coordinate
$z$ on $S$ such that $z=0$ is $\tb_i$ and a holomorphic coordinate $w$ on
$T$ such that $w=0$ is $b_i$ and such that $\pi^{\ast}(w) = z^{e_i}$.
This implies that $\pi^{\ast}(\dw) = e_i z^{e_i-1} \dz$, so for $k \in \Z$ we have
\[\pi^{\ast}(w^k \dw^2) = z^{e_i k} e_i^2 z^{2e_i-2} \dz^2 = e_i^2 z^{e_i(k+2)-2} \dz^2.\]
There are now two cases:
\begin{itemize}
\item If $\tb_i \notin \{p_1,\ldots,p_n\}$, then elements of $Q(S;p_1,\ldots,p_n)$
must be holomorphic at $\tb_i$.  For $\pi^{\ast}(w^k \dw^2)$ to be holomorphic at $\tb_i$,
we must have $e_i(k+2) \geq 2$.  Since $e_i \geq 2$, this will hold exactly when
$k \geq -1$.
\item If $\tb_i \in \{p_1,\ldots,p_n\}$, then elements of $Q(S;p_1,\ldots,p_n)$
are allowed to have at most a simple pole at $\tb_i$.  For $\pi^{\ast}(w^k \dw^2)$ to
have at most a simple pole at $\tb_i$, we must have $e_i(k+2) \geq 1$.  Since $e_i \geq 2$, this will hold exactly when
$k \geq -1$.
\end{itemize}
The step follows from the above two calculations.  We remark that these
calculations also show that an $f$-invariant element of
$Q(S;p_1,\ldots,p_n)$ is holomorphic at each $p_i$.
Consequently, $Q(S;p_1,\ldots,p_n)^f = Q(S)^f$.  Throughout the rest of the
proof, we will therefore just talk about $Q(S)^f$.

\begin{step}{2}
\label{step.invariantdifferentials.2}
We prove that $r \leq 2g-4h+2$.
\end{step}

As we stated before Step \ref{step.invariantdifferentials.1}, the Riemann--Hurwitz formula
says that
\[2g-2 = d\left(2h-2 + \sum_{i=1}^r\left(1-\frac{1}{e_i}\right)\right).\]
Since $e_i \geq 2$ for all $1 \leq i \leq r$, we have
\[\sum_{i=1}^r\left(1-\frac{1}{e_i}\right) \geq \frac{r}{2}.\]
Substituting this into the Riemann--Hurwitz formula, we get
$2g-2 \geq d\left(2h-2+ \frac{r}{2}\right)$.
Rearranging this and using the fact that $d \geq 2$, we get
\[r \leq \frac{4g-4}{d} - 4h + 4 \leq (2g-2)-4h+4 = 2g-4h+2.\]

\begin{step}{3}
\label{step.invariantdifferentials.3}
We complete the proof in the case where either $g+n \geq 4$ or $g+n=3$ and $h \geq 1$.
\end{step}

Recall that our goal is to prove that the dimension of $Q(S)^f$ is at most $3g-5+n$.
In Step \ref{step.invariantdifferentials.1} we proved that the dimension of
$Q(S)^f$ is $3h-3+r$.  Combining this with the inequality $r \leq 2g-4h+2$
from Step \ref{step.invariantdifferentials.2}, we see that the dimension
of $Q(S)^f$ is at most $3g-5+n$ if
\[3h-3+(2g-4h+2) \leq 3g-5+n.\]
Rearranging this, we see that this is equivalent to $g+n+h \geq 4$.  This
inequality holds if either $g+n \geq 4$ or $g+n=3$ and $h \geq 1$, as desired.

\begin{step}{4}
\label{step.invariantdifferentials.4}
We complete the proof when $(g,n) = (3,0)$ and $h=0$.
\end{step}

In this case, our assumptions say that the automorphism $f\colon S \rightarrow S$ is not
the hyperelliptic involution.  Since $T = S/f$ has genus $h=0$, this means that
the projection $\pi\colon S \rightarrow T = \bbP^1$ cannot have degree $2$.
We therefore have $d \geq 3$.

Step \ref{step.invariantdifferentials.1} says that the dimension of
$Q(S)^f$ is $3h-3+r = r-3$, and we must prove that this is at most $3g-5+n = 4$.  In other words,
we must prove that $r \leq 7$.  Unfortunately, the bound from
Step \ref{step.invariantdifferentials.2} is only $r \leq 8$.  Revisiting
the proof of Step \ref{step.invariantdifferentials.2}, its final line says
that
\[r \leq \frac{4g-4}{d} - 4h + 4 = \frac{8}{d} + 4 .\]
Since $d \geq 3$, the right-hand side of this is at most $8/3+4 \leq 7$,
as desired.

\begin{step}{5}
\label{step.invariantdifferentials.5}
We complete the proof when $(g,n) = (2,1)$ and $h=0$.
\end{step}

Again, our assumptions say that $f$ is not the hyperelliptic involution, so $d \geq 3$.
Step \ref{step.invariantdifferentials.1} says that the dimension of
$Q(S)^f$ is $3h-3+r = r-3$, and we must prove that this is at most $3g-5+n = 2$.  In other words,
we must prove that $r \leq 5$.  The final line of the proof of Step \ref{step.invariantdifferentials.2}
together with the fact that $d \geq 3$ implies that
\[r \leq \frac{4g-4}{d} - 4h + 4 = \frac{4}{d}+4 \leq \frac{4}{3}+4 = 5+\frac{1}{3}.\]
Since $r$ is an integer, this implies that $r \leq 5$, as desired.

\begin{step}{6}
\label{step.invariantdifferentials.6}
We complete the proof when $(g,n) = (1,2)$ and $h=0$.
\end{step}

Again, our assumptions say that $f$ is not the hyperelliptic involution.
Step \ref{step.invariantdifferentials.1} says that the dimension of
$Q(S)^f$ is $3h-3+r = r-3$, and we must prove that this is at most $3g-5+n = 0$.  In other words,
we must prove that $r \leq 3$.  The Riemann--Hurwitz formula in this case is
\[0 = d\left(-2 + \sum_{i=1}^r\left(1-\frac{1}{e_i}\right)\right).\]
Dividing by $d$ and rearranging, we get
\[\sum_{i=1}^r\left(1-\frac{1}{e_i}\right) = 2.\]
Since each $e_i$ has $e_i \geq 2$, we must have $r \leq 4$.  To rule out the
case $r=4$, note that if $r=4$ then each $e_i$ equals $2$.  The order $d$ of the
automorphism $f\colon S \rightarrow S$ must equal the ramification index
at one of the two marked fixed points, so we deduce that $d=2$.  This implies
that $f$ is the hyperelliptic involution, contrary to our assumptions.
\end{proof}

\subsection{Locus of curves with automorphisms, edge cases}

We now turn to the edge cases $\Moduli_2$ and $\Moduli_{1,1}$.
In these cases, every point in these moduli spaces represents a hyperelliptic curve,
which thus has an automorphism group of order at least $2$.
This is why we defined $\Moduli_2^{\sym}$ and $\Moduli_{1,1}^{\sym}$
to be the loci of points in these moduli spaces corresponding to
curves whose automorphism group has order greater than $2$.  Just
like for the other cases, this is a proper Zariski closed subspace:

\begin{lemma}
\label{lemma:symclosedlowgenus}
The following hold:
\begin{itemize}
\item[(i)] The subspace $\Moduli_2^{\sym}$ is a proper Zariski
closed subspace of $\Moduli_2$.
\item[(ii)] The subspace $\Moduli_{1,1}^{\sym}$ is a proper
Zariski closed subspace of $\Moduli_{1,1}$.
\end{itemize}
\end{lemma}
\begin{proof}
For (i), recall from Example \ref{example:moduli2}
that $\Moduli_2 = \Moduli_{0,6}/S_6$.  
The locus $\Moduli_2^{\sym}$ is the image of the subspace of $\Moduli_{0,6}$ consisting of points
that are fixed by some nontrivial element of $S_6$.  Since $S_6$ acts
algebraically on $\Moduli_{0,6}$, this locus is a Zariski closed
subspace $Z \subset \Moduli_{0,6}$.  It is proper since a generic
set of $6$ points on $\bbP^1$ cannot be nontrivially permuted by a M\"{o}bius
transformation.  The map $\Moduli_{0,6} \rightarrow \Moduli_2$
is a finite map, so the image $\Moduli_2^{\sym}$ of $Z$ is also a proper
Zariski closed subspace (see \S \ref{section:finitemaps}).

The proof of (ii) is similar, using the fact
that $\Moduli_{1,1} \cong \Moduli_{0,4}/S_3$ (see Example \ref{example:moduli11}).
\end{proof}

\begin{remark}
Igusa \cite{IgusaGenus2} proved that the singular locus of $\Moduli_2$ consists
of the single point representing
the normalization of the projective closure of the plane curve $y^2=x^5-1$. 
\end{remark}

\begin{remark}
Another way to proceed for $\Moduli_{1,1}$ is to view it as $\bbH^2/\PSL_2(\Z)$.
The group $\PSL_2(\Z)$ has two conjugacy classes of nontrivial finite subgroups,
one cyclic of order $2$ and the other cyclic of order $3$.  From this, we see that
$\Moduli_{1,1}^{\sym}$ consists of $2$ points, namely the images of
the fixed points of representatives of these two conjugacy classes of finite subgroups.
From this point of view we can also see that $\Moduli_{1,1}$ is smooth,
so its singular locus is empty.
\end{remark}

\section{Finite covers of the moduli space of curves}
\label{section:modulicover}

Fix $g,n \geq 0$ such that $\Sigma_{g,n}$ is hyperbolic.
The moduli space $\Moduli_{g,n}$ is typically not smooth.
In this section, we construct finite ramified covers of $\Moduli_{g,n}$ that are smooth.

\subsection{Covers of moduli space}

To clarify what we are doing, we will temporarily again denote the analytification of
$\Moduli_{g,n}$ by $\Moduli_{g,n}^{\an}$.  Recall that $\PMod_{g,n}$ acts on
the Teichm\"{u}ller space $\Teich_{g,n}$ and $\Moduli_{g,n}^{\an} = \Teich_{g,n}/\PMod_{g,n}$.
For a finite-index subgroup $G < \PMod_{g,n}$, let $\Moduli_{g,n}(G,\an) = \Teich_{g,n}/G$.
This is an analytic space, and the projection $\Moduli_{g,n}(G,\an) \rightarrow \Moduli_{g,n}^{\an}$
is a finite map.  We will soon prove that the projection $\pi\colon \Moduli_{g,n}(G,\an) \rightarrow \Moduli_{g,n}^{\an}$
has a unique algebraization, but first we record some simple properties of $\Moduli_{g,n}(G,\an)$:

\begin{lemma}
\label{lemma:coverproperties}
Fix $g,n \geq 0$ such that $\Sigma_{g,n}$ is hyperbolic.  Let $G < \PMod_{g,n}$
be a finite-index subgroup.  Assume that $G$ is torsion-free.  Then the following
hold:
\begin{itemize}
\item[(i)] The space $\Moduli_{g,n}(G,\an)$ is a smooth complex manifold.
\item[(ii)] The space $\Moduli_{g,n}(G,\an)$ is a $K(G,1)$.
\end{itemize}
\end{lemma}
\begin{proof}
An element of $\PMod_{g,n}$ fixes a point of $\Teich_{g,n}$ if and only if it has finite order.
Since $G < \PMod_{g,n}$ is torsion-free, it therefore acts freely on Teichm\"{u}ller space.  
Since Teichm\"{u}ller space is a complex manifold on which $\PMod_{g,n}$ acts properly discontinuously, this implies (i).
Since Teichm\"{u}ller space is contractible, it also implies (ii).
\end{proof}

To apply Lemma \ref{lemma:coverproperties}, we need a source of finite-index torsion-free subgroups
of $\PMod_{g,n}$.  For $\ell \geq 2$, the {\em level-$\ell$} subgroup of $\PMod_{g,n}$
is\footnote{Sometimes
the level-$\ell$ subgroup is the kernel of the action of $\PMod_{g,n}$ on $\HH_1(\Sigma_g;\Z/\ell)$
obtained by filling in the punctures, but the definition we give works better
in this particular context.}
\[\PMod_{g,n}(\ell) = \Set{$f \in \PMod_{g,n}$}{$f$ acts trivially on $\HH_1(\Sigma_{g,n};\Z/\ell)$}.\]
This is a finite-index normal subgroup of $\PMod_{g,n}$, and we have:

\begin{lemma}[{\cite[Corollary 7.1.B]{IvanovSurvey}}]
\label{lemma:leveltorsionfree}
Fix $g,n \geq 0$ such that $\Sigma_{g,n}$ is hyperbolic.  For $\ell \geq 3$,
the group $\PMod_{g,n}(\ell)$ is torsion-free.
\end{lemma}

We now prove the following.  See \S \ref{section:normaldef} for
the definition of a normal variety.

\begin{lemma}
\label{lemma:modulinormal}
Fix $g,n \geq 0$ such that $\Sigma_{g,n}$ is hyperbolic.  Let $G < \PMod_{g,n}$
be a finite-index subgroup.  Then $\Moduli_{g,n}(G,\an)$ is an irreducible normal analytic variety.
In particular, $\Moduli_{g,n}$ is irreducible and normal.
\end{lemma}
\begin{proof}
Set $G' = G \cap \PMod_{g,n}(3)$, so by Lemma \ref{lemma:leveltorsionfree} the group $G'$ is a torsion-free finite-index
normal subgroup of $G$.  Lemma \ref{lemma:coverproperties} says that
$\Moduli_{g,n}(G',\an)$ is a smooth analytic variety.  Setting $Q = G/G'$, we have
$\Moduli_{g,n}(G,\an) = \Moduli_{g,n}(G',\an)/Q$.  Since $\Moduli_{g,n}(G,\an)$ is the quotient
of a smooth analytic variety by a finite group, it follows that $\Moduli_{g,n}(G,\an)$ is
normal (see Lemma \ref{lemma:normalquotient}).  Since $\Moduli_{g,n}(G,\an)$ is a quotient
of the connected analytic variety $\Teich_{g,n}$, the normal variety $\Moduli_{g,n}(G,\an)$
is connected and hence irreducible.
\end{proof} 

\subsection{Algebraicity}

The following theorem first appeared in Mumford's paper \cite{MumfordAbelian}, at
least for $n=0$:

\begin{theorem}[{Mumford \cite{MumfordAbelian}}]
\label{theorem:modulicoveralg}
Fix $g,n \geq 0$ such that $\Sigma_{g,n}$ is hyperbolic.  Let $G < \PMod_{g,n}$
be a finite-index subgroup.  Then the finite map $\pi\colon \Moduli_{g,n}(G,\an) \rightarrow \Moduli_{g,n}^{\an}$
has a unique algebraization.
\end{theorem}

The paper \cite{MumfordAbelian} says that Theorem \ref{theorem:modulicoveralg} follows
from its version of the generalized Riemann existence theorem, which we discussed
earlier (see Theorem \ref{theorem:mumfordriemann}).  It does not indicate how
to verify the hypotheses of that theorem, so for completeness we include the proof.

\begin{proof}[Proof of Theorem \ref{theorem:modulicoveralg}]
As a warm-up, we prove it in genus $0$:

\begin{claim}{1}
\label{claim:modulicover.1}
Fix $n \geq 3$.  Let $G < \PMod_{0,n}$ be a finite-index subgroup.  
The finite map $\pi\colon \Moduli_{0,n}(G,\an) \rightarrow \Moduli_{0,n}^{\an}$
has a unique algebraization.
\end{claim}
\begin{proof}[Proof of claim]
The key point here is that $\PMod_{0,n}$ is torsion-free, so it acts
freely on Teichm\"{u}ller space.  It follows that the map $\Teich_{0,n} \rightarrow \Moduli_{0,n}^{\an}$
is an unbranched cover, and thus that the map $\pi\colon \Moduli_{0,n}(G,\an) \rightarrow \Moduli_{0,n}^{\an}$
is a finite unbranched cover.  The usual generalized Riemann existence theorem 
(Theorem \ref{theorem:generalizedriemann})
therefore implies that the map $\Moduli_{0,n}(G,\an) \rightarrow \Moduli_{0,n}^{\an}$ has a unique algebraization, as desired.
\end{proof}

For $g \geq 1$, we first handle the cases where $g+n \geq 3$ and then adjust the proof
to deal with the remaining cases $(g,n) = (2,0)$ and $(g,n) = (1,1)$.

\begin{claim}{2}
\label{claim:modulicover.2}
Fix $g \geq 1$ and $n \geq 0$ such that $g+n \geq 3$.
Let $G < \PMod_{g,n}$ be a finite-index subgroup.  
The finite map $\pi\colon \Moduli_{g,n}(G,\an) \rightarrow \Moduli_{g,n}^{\an}$
has a unique algebraization.
\end{claim}

We will use Mumford's version of the generalized 
Riemann existence theorem (Theorem \ref{theorem:mumfordriemann}).
This requires verifying its hypotheses, which in our setting are as follows:
\begin{itemize}
\item[(a)] Both $\Moduli_{g,n}(G,\an)$ and $\Moduli_{g,n}$ are irreducible, which follows from Lemma \ref{lemma:modulinormal}.
\item[(b)] The analytic variety $\Moduli_{g,n}(G,\an)$ is normal, which also follows from Lemma \ref{lemma:modulinormal}.
\item[(c)] There exists a Zariski open set $U \subset \Moduli_{g,n}$ such that
$V = \pi^{-1}(U^{\an})$ is connected and dense in $\Moduli_{g,n}(G,\an)$ and
$\pi\colon V \rightarrow U^{\an}$ is a finite unbranched cover.  We will verify this
below.
\end{itemize}
It remains to construct a Zariski open set $U \subset \Moduli_{g,n}$ as in (c).  

Set $Z = \Moduli_{g,n}^{\sym}$.  Corollary \ref{corollary:symclosed} says that $Z$
is a proper Zariski closed subset of $\Moduli_{g,n}$.  Setting $U = \Moduli_{g,n} \setminus Z$,
it follows that $U$ is a nonempty Zariski open set. 
Before studying $V = \pi^{-1}(U^{\an})$, we make some preliminary observations.
Let $\tZ$ and $\tU$ be the preimages in the Teichm\"{u}ller space
$\Teich_{g,n}$ of $Z^{\an}$ and $U^{\an}$, respectively.
Since $\tZ$ is the preimage of the analytification of
the proper Zariski closed subset $Z \subset \Moduli_{g,n}$, it follows
that $\tZ$ is a proper analytic subspace of $\Teich_{g,n}$.  Since
$\tU = \Teich_{g,n} \setminus \tZ$, we deduce that $\tU$ is connected\footnote{The point here is that proper analytic subspaces of complex manifolds have real codimension at least $2$,
so they cannot disconnect the manifold.}
and dense.

We claim that $\PMod_{g,n}$ acts freely on $\tU$.
Indeed, the only elements of $\PMod_{g,n}$ that fix points of $\Teich_{g,n}$ are the finite-order ones.
If $f \in \PMod_{g,n}$ is a nontrivial finite-order element fixing $x \in \Teich_{g,n}$, then the
image of $x$ in $\Moduli_{g,n}$ is a point $\cS$ whose automorphism group contains an element
corresponding to $f$.  Since $Z = \Moduli_{g,n}^{\sym}$, 
it follows that $\cS \in Z^{\an}$, so $x \in \tZ = \Teich_{g,n} \setminus \tU$.
The claim follows.

Since $\PMod_{g,n}$ acts freely and properly discontinuously on $\tU$ and 
$U^{\an} = \tU/\PMod_{g,n}$, the map
$\tU \rightarrow U^{\an}$ is an unbranched holomorphic covering space.
Set $V = \pi^{-1}(U^{\an})$.  Since $V \subset \Moduli_{g,n}(G,an)$ is the image of the connected
dense subset $\tU \subset \Teich_{g,n}$, it follows that $V$ is connected and dense.  
We have $V \cong \tU/G$, and the covering space $\tU \rightarrow U^{\an}$ factors as
\[\begin{tikzcd}
\tU \arrow[rr, bend right=50, "/\PMod_{g,n}"{yshift=2pt}] \arrow{r}{/G} & V \arrow{r}{\pi} & U^{\an}.
\end{tikzcd}\]
Since $G$ is a finite-index subgroup of $\PMod_{g,n}$, we deduce that
$\pi\colon V \rightarrow U^{\an}$ is a finite unbranched cover.  The
hypotheses of Theorem \ref{theorem:mumfordriemann} are therefore satisfied.

\begin{claim}{3}
\label{claim:modulicover.3}
Let $(g,n) \in \{(2,0),(1,1)\}$.
Let $G < \PMod_{g,n}$ be a finite-index subgroup.
The finite map $\pi\colon \Moduli_{g,n}(G,\an) \rightarrow \Moduli_{g,n}^{\an}$
has a unique algebraization.
\end{claim}

In this case, the group $\PMod_{g,n}$ has a nontrivial
center $Z(\PMod_{g,n})$ which acts trivially on $\Teich_{g,n}$.  This center
is the cyclic group generated by the hyperelliptic involution.  
The action of $\PMod_{g,n}$ on $\Teich_{g,n}$ factors through
$\Gamma = \PMod_{g,n} / Z(\PMod_{g,n})$.  Letting $\oG < \Gamma$ be the image
of $G < \PMod_{g,n}$ under the projection $\PMod_{g,n} \rightarrow \Gamma$,
we have $\Moduli_{g,n}(G,\an) = \Teich_{g,n}/\oG$.  

The automorphism group of each $\cS \in \Moduli_{g,n}$ contains the hyperelliptic involution.
A point $x \in \Teich_{g,n}$ is fixed by a nontrivial finite-order element $f \in \Gamma$
exactly when the image $\cS \in \Moduli_{g,n}$ of $x$ has an automorphism not contained
in the cyclic group generated by the hyperelliptic involution.
Using the $\Gamma$-action instead of the $\PMod_{g,n}$-action, the proof
proceeds exactly like in Claim \ref{claim:modulicover.2}, with
Lemma \ref{lemma:symclosedlowgenus} substituting for Corollary \ref{corollary:symclosed}.
\end{proof}

\subsection{Normality of algebraization}

Let $G < \PMod_{g,n}$ be a finite-index subgroup.  Theorem \ref{theorem:modulicoveralg}
says that the finite map $\pi\colon \Moduli_{g,n}(G,\an) \rightarrow \Moduli_{g,n}^{\an}$
has a unique algebraization.  We will denote the resulting algebraization of
$\Moduli_{g,n}(G,\an)$ by $\Moduli_{g,n}(G)$.  Just like for $\Moduli_{g,n}$, to
simplify our notation we will not distinguish between $\Moduli_{g,n}(G)$ and
$\Moduli_{g,n}(G,\an)$ unless it is needed for clarity.  We record the following
consequence of our previous results:

\begin{lemma}
\label{lemma:modgsmoothnormal}
Fix $g,n \geq 0$ such that $\Sigma_{g,n}$ is hyperbolic.  Let $G < \PMod_{g,n}$
be a finite-index subgroup.  Then $\Moduli_{g,n}(G)$ is a normal variety, and
is smooth if $G$ is torsion-free.
\end{lemma}
\begin{proof}
Since an algebraic variety is normal if its analytification is normal, this
is immediate from Lemmas \ref{lemma:modulinormal} and \ref{lemma:coverproperties}.
\end{proof}

\section{The Deligne--Mumford compactification}
\label{section:delignemumford}

Fix $g,n \geq 0$ such that $\Sigma_{g,n}$ is hyperbolic.  At least for $n=0$, Baily \cite{BailyQuasi} proved
that $\Moduli_{g,n}$ is a quasiprojective variety.  The points Baily added to $\Moduli_{g,n}$ to form
his projective compactification did not have a clear geometric meaning.  Later, Deligne and Mumford
constructed a geometrically meaningful compactification of $\Moduli_{g,n}$ called the Deligne--Mumford
compactification.  This section reviews its properties.

\subsection{References}

The most natural approach to the Deligne--Mumford compactification
for someone interested in mapping class groups
uses augmented Teichm\"{u}ller space.  This
has a large literature that we will not try to summarize.
We use Hubbard--Koch \cite{HubbardKoch}
as our primary reference, which gives a complete
construction of the Deligne--Mumford compactification using
Teichm\"{u}ller theory.  Hinich--Vaintrob \cite{HinichVaintrob} is
a good secondary reference, but it assumes that the Deligne--Mumford
compactification is already known to exist.

\subsection{Stable topological surfaces}
\label{section:stablesurface}

Except in the trivial case $(g,n) = (0,3)$ where it is a single point, 
the moduli space $\Moduli_{g,n}$ is not compact, so it is not a projective 
variety.  Deligne and Mumford compactified $\Moduli_{g,n}$ by adding points
corresponding to certain singular algebraic curves.  Before describing this,
we introduce the topological models of these singular curves.  

Recall that a simple closed curve on $\Sigma_{g,n}$ is {\em essential} if
it does not bound a disk or a once-punctured disk.  Let
$\sigma = \{\gamma_1,\ldots,\gamma_k\}$ be a (possibly empty) collection of 
disjoint essential simple closed curves that are pairwise non-isotopic.
If $k \geq 1$, then the isotopy classes of the $\gamma_i$ form a $(k-1)$-simplex in
the curve complex $\cC_{g,n}$. 

Define $\Sigma_{g,n}/\sigma$
to be the space obtained by collapsing each $\gamma_i$ to a node $x_i$.  See here:
\Figure{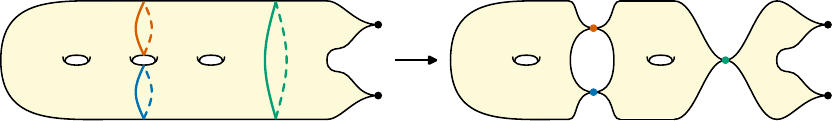}
Spaces of this form are called {\em stable topological surfaces}.  Define
\[\Sigma_{g,n}[\sigma] = \Sigma_{g,n} \setminus \bigcup_{i=1}^k \gamma_i \cong \Sigma_{g,n}/\sigma \setminus \{x_1,\ldots,x_k\}.\]
This is a disjoint union of punctured hyperbolic surfaces.  

Let $\PMod[\sigma]$ be the product of the pure\footnote{Here ``pure'' means
that they fix each puncture, including the two punctures coming from each $\gamma \in \sigma$.}
mapping class groups of the components of $\Sigma_{g,n}[\sigma]$.
Elements of this product extend over the $x_i$ to mapping classes on the
stable surface $\Sigma_{g,n}/\sigma$.
Let $\PMod\Span{\sigma}$ be the subgroup of $\PMod_{g,n}$ consisting of mapping classes
that take each circle $\gamma \in \sigma$ to itself by an orientation-preserving homeomorphism.
Note that this makes sense even though we have not chosen an orientation on $\gamma$.
There is a surjective map $\PMod\Span{\sigma} \rightarrow \PMod[\sigma]$ whose
kernel is the free abelian group generated by the Dehn twists about the curves in $\sigma$:
\[1 \rightarrow \SpanSet{$T_{\gamma}$}{$\gamma \in \sigma$} \rightarrow \PMod\Span{\sigma} \rightarrow \PMod[\sigma] \rightarrow 1.\]
Compare \cite[Proposition 2.5]{HubbardKoch}.

\subsection{Compactification}
\label{section:dm}

The Deligne--Mumford compactification $\oModuli_{g,n}$ is a projective variety containing
$\Moduli_{g,n}$ as a dense Zariski open subspace \cite{DeligneMumford, Knudsen2, Knudsen3}.
Since $\Moduli_{g,n}$ is irreducible (see Lemma \ref{lemma:modulinormal}) and
dense in $\oModuli_{g,n}$, it follows that $\oModuli_{g,n}$ is irreducible.
Points of $\oModuli_{g,n}$ parameterize isomorphism classes of tuples $\cS = (S;p_1,\ldots,p_n)$ where:
\begin{itemize}
\item $S$ is a connected projective algebraic curve of arithmetic genus $g$ whose only
singularities are nodes.  Here the arithmetic genus of a nodal curve is the genus of the smooth curve obtained
by smoothing each node.
\item The $p_i$ are distinct points in the smooth locus of $S$.
\item $\Aut(\cS)$ is a finite group.  Equivalently, the following holds for each component $T$ of the normalization\footnote{Topologically, this
normalization separates the two branches at each node of $S$ to produce a possibly disconnected smooth topological surface.} of $S$ such that
$T \cong \bbP^1$.  Let $Q \subset T$ be the set of points that map to either points of $\{p_1,\ldots,p_n\} \subset S$
or to nodes of $S$.  We remark that two points of $Q$ might map to a single node in $S$.  Then $|Q| \geq 3$.
\end{itemize}
As topological spaces
with marked points, these stable curves are exactly those that are 
obtained from a stable topological surface $\Sigma_{g,n}/\sigma$ for
some (possibly empty) $\sigma$ by replacing each of the $n$ original punctures
with a marked point.  The fact that $\Aut(\cS)$ is finite ensures that
the curves in $\sigma$ are essential and pairwise non-isotopic.

\subsection{Clutching maps}
\label{section:clutching1}

A reference for the material in this subsection is \cite[\S 3]{Knudsen2}.
Let $\sigma$ be either $\emptyset$ or a simplex of $\cC_{g,n}$.  Define $\Moduli[\sigma]$ to
be the product of the moduli spaces of curves associated to each component
of $\Sigma_{g,n}[\sigma]$.  There is an algebraic map $\Moduli[\sigma] \rightarrow \oModuli_{g,n}$
called a {\em clutching map} that glues the two marked points corresponding to each $\gamma \in \sigma$ to form a node.
The image of the clutching map consists of stable curves whose underlying stable
topological surface is $\Sigma_{g,n}/\sigma$:
\Figure{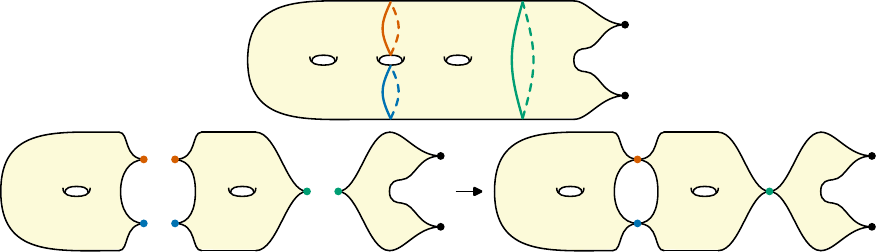}
The top of this figure shows a simplex $\sigma$, and the bottom row shows the components
of $\Sigma_{g,n}[\sigma]$ being glued together to form the stable surface $\Sigma_{g,n}/\sigma$.

Let $\oModuli[\sigma]$ be the product of the Deligne--Mumford compactifications
of the moduli spaces of curves associated to each component of $\Sigma_{g,n}[\sigma]$.
The clutching map $\Moduli[\sigma] \rightarrow \oModuli_{g,n}$ extends
to a clutching map $\oModuli[\sigma] \rightarrow \oModuli_{g,n}$.  The clutching
map $\oModuli[\sigma] \rightarrow \oModuli_{g,n}$ is finite.

Let $\Delta[\sigma] \subset \oModuli_{g,n}$
be the image of the clutching map $\oModuli[\sigma] \rightarrow \oModuli_{g,n}$.  The
subspace $\Delta[\sigma]$ is the closure of the image of $\Moduli[\sigma]$.  In this closure,
curves with additional nodes may appear.
The variety $\oModuli[\sigma]$ is a product of irreducible projective varieties,
so $\oModuli[\sigma]$ is also irreducible and projective.  Since the clutching map
$\oModuli[\sigma] \rightarrow \oModuli_{g,n}$ is a finite map, it follows
that its image $\Delta[\sigma]$ is an irreducible projective subvariety of $\oModuli_{g,n}$ (see \S \ref{section:finitemaps}).

\begin{remark}
For a simplex $\sigma$ of $\cC_{g,n}$, the clutching maps 
$\Moduli[\sigma] \rightarrow \oModuli_{g,n}$ and $\oModuli[\sigma] \rightarrow \oModuli_{g,n}$ 
are typically not injective.
\end{remark}

\begin{remark}
\label{remark:pmodclutch}
For $f \in \PMod_{g,n}$, we have
$\Delta[\sigma] = \Delta[f(\sigma)]$.  In other words, the images of
the clutching maps are really parameterized by $\PMod_{g,n}$-orbits of simplices
of $\cC_{g,n}$.
\end{remark}

\subsection{Boundary of the compactification}

Let $\partial \oModuli_{g,n} = \oModuli_{g,n} \setminus \Moduli_{g,n}$.  By construction,
each point of $\partial \oModuli_{g,n}$ lies in $\Delta[\sigma]$
for some simplex $\sigma$ of $\cC_{g,n}$.  Moreover, if $\sigma$ is a face of $\sigma'$ then
the clutching map $\oModuli[\sigma'] \rightarrow \oModuli_{g,n}$ factors
as
\[\oModuli[\sigma'] \rightarrow \oModuli[\sigma] \rightarrow \oModuli_{g,n},\]
where the second map is a clutching map and the first map is a partial clutching map\footnote{To help understand
this, note that since $\sigma \subset \sigma'$ we can form $\Sigma_{g,n}[\sigma']$ by
first deleting the curves in $\sigma$ and then deleting the curves in $\sigma' \setminus \sigma$.  Each
component of $\Sigma_{g,n}[\sigma]$ therefore
breaks into one or more components of $\Sigma_{g,n}[\sigma']$.} 
that glues together the marked points corresponding to the curves in $\sigma' \setminus \sigma$. 
It follows that $\partial \oModuli_{g,n}$ is the union of the $\Delta[\sigma]$
as $\sigma$ ranges over all vertices of $\cC_{g,n}$, i.e., over essential simple closed
curves.

\subsection{Boundary divisors}
\label{section:boundarydivisors}

For an essential simple closed curve $\gamma$, let $\Delta[\gamma] = \Delta[\{\gamma\}]$.
As we noted in Remark \ref{remark:pmodclutch}, if $\gamma$ and $\gamma'$
are essential simple closed curves such that there exists some $f \in \PMod_{g,n}$
with $f(\gamma) = \gamma'$, then $\Delta[\gamma] = \Delta[f(\gamma)]$.
By the change of coordinates principle from \cite[\S 1.3]{FarbMargalitPrimer},
the $\PMod_{g,n}$-orbits of essential simple closed curves fall into the following
families.

First, for $g \geq 1$ we have the orbit of a nonseparating simple closed curve $\gamma$.  For
such a $\gamma$, we have $\oModuli[\gamma] \cong \oModuli_{g-1,n+2}$:
\Figure{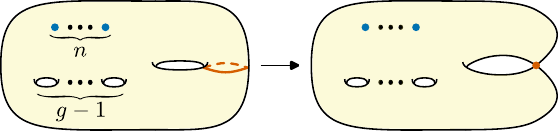}
The image $\Delta[\gamma]$ of the corresponding clutching map $\oModuli[\gamma] \rightarrow \oModuli_{g,n}$
is written $\Delta_{\irr}$.  It is an irreducible projective subvariety of
$\partial \oModuli_{g,n}$.

Second, for each $0 \leq i \leq g$ and each subset $P \subset \{1,\ldots,n\}$ such that
$\Sigma_{i,|P|+1}$ and $\Sigma_{g-i,n-|P|+1}$ are both hyperbolic surfaces, we have
the orbit of a simple closed separating curve that divides the surface into a genus-$i$ subsurface
containing the punctures $P$ and a genus-$(g-i)$ subsurface containing the punctures
$P^c=\{1,\ldots,n\} \setminus P$.  For such a $\gamma$, we have $\oModuli[\gamma] \cong \oModuli_{i,|P|+1} \times \oModuli_{g-i,n-|P|+1}$:
\Figure{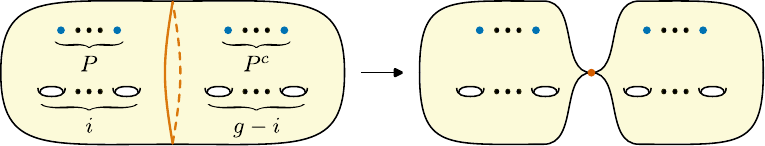}
The image $\Delta[\gamma]$ of the corresponding clutching map $\oModuli[\gamma] \rightarrow \oModuli_{g,n}$
is written $\Delta_{i,P}$.  It is an irreducible projective subvariety of
$\partial \oModuli_{g,n}$.  Since the two sides of $\gamma$ are unordered, we
have $\Delta_{i,P} = \Delta_{g-i,P^c}$.

For an irreducible variety $X$, an {\em irreducible divisor} in $X$ is a closed irreducible
subvariety $Y \subset X$ of codimension $1$.  Both $\Delta_{irr}$ and the $\Delta_{i,P}$ are
irreducible divisors in $\Moduli_{g,n}$.  Indeed, $\dim(\Moduli_{g,n}) = 3g-3+n$ and
\begin{align*}
\dim(\oModuli_{g-1,n+2}) &= 3(g-1)-3+(n+2) = (3g-3+n)-1,\\
\dim(\oModuli_{i,|P|+1} \times \oModuli_{g-i,n-|P|+1}) &= (3i-3+|P|+1) + (3(g-i)-3+n-|P|+1) \\
&= (3g-3+n)-1.
\end{align*}
We call the divisors $\Delta_{irr}$ and $\Delta_{i,P}$ the {\em boundary divisors} of
$\Moduli_{g,n}$.  

The boundary $\partial \oModuli_{g,n}$ is the union of $\Delta_{irr}$ (which only appears when $g \geq 1$) and the distinct $\Delta_{i,P}$.  Since
these are closed irreducible subvarieties of $\partial \oModuli_{g,n}$ and each contains a point
not in any of the others, they are in fact the irreducible components of $\partial \oModuli_{g,n}$.
It follows that $\partial \oModuli_{g,n}$ is a divisor in $\oModuli_{g,n}$, that is, a closed
subvariety each of whose irreducible components is a divisor.

\subsection{Analytic construction of the Deligne--Mumford compactification}
\label{section:stratainfinity}

Completing a program begun by Abikoff \cite{AbikoffAugmented} and continued
by many people, Hubbard--Koch \cite{HubbardKoch} showed how to construct $\oModuli_{g,n}$
analytically.  There is an augmented Teichm\"{u}ller space $\oTeich_{g,n}$ containing
$\Teich_{g,n}$ as a dense open subspace.  Just like for $\Teich_{g,n}$, the space
$\oTeich_{g,n}$ is contractible.\footnote{This result does not appear in \cite{HubbardKoch}.  An alternate
description of $\oTeich_{g,n}$ is that it is the metric completion of the Weil--Petersson metric on
Teichm\"{u}ller space (see \cite{Wolpert} for a survey).  For $n=0$, the paper \cite{Yamada} proves that this
metric completion is CAT(0), and hence in particular is contractible.  See \cite[Theorem 14]{Wolpert} for
the extension of this to general $(g,n)$.}  
The group $\PMod_{g,n}$ acts on $\oTeich_{g,n}$,
and as a topological space $\oModuli_{g,n}^{\an}$ is homeomorphic to $\oTeich_{g,n} / \PMod_{g,n}$.
Points of $\partial \oTeich_{g,n} = \oTeich_{g,n} \setminus \Teich_{g,n}$
correspond to isotopy classes of complete hyperbolic metrics of finite
volume on surfaces of the form $\Sigma_{g,n}[\sigma]$
for some simplex $\sigma$ of $\cC_{g,n}$.
   
One subtle issue here is that except in the trivial case $(g,n)=(0,3)$, the
space $\oTeich_{g,n}$ is {\em not} a complex manifold.  In
fact, it is not even locally compact.  However, \cite{HubbardKoch} shows that
there is a natural way to endow $\oTeich_{g,n} / \PMod_{g,n}$ with the structure
of a complex analytic space, and we have an isomorphism
$\oModuli_{g,n}^{\an} \cong \oTeich_{g,n} / \PMod_{g,n}$
of complex analytic spaces \cite[Theorem 13.1]{HubbardKoch}.
Just like for $\Moduli_{g,n}$, to simplify our notation
we will not distinguish between $\oModuli_{g,n}$ and $\oModuli_{g,n}^{\an}$
unless it is necessary for clarity.

\subsection{Stratification}

In \cite[\S 2.1]{HubbardKoch}, the space $\oTeich_{g,n}$ is given a stratification
by subspaces that are products of simpler Teichm\"{u}ller spaces.  Our next
goal is to describe this stratification.

Let $\sigma$ be a (possibly empty) collection of essential simple closed
curves on $\Sigma_{g,n}$ that are pairwise disjoint and non-isotopic.  If $\sigma \neq \emptyset$,
then $\sigma$ is a simplex of $\cC_{g,n}$.  Associated
to $\sigma$ is a locally closed subset $\Teich[\sigma]$ of $\oTeich_{g,n}$ called
a {\em stratum}.  We have $\Teich[\emptyset] = \Teich_{g,n}$, and more generally
$\Teich[\sigma]$ parameterizes isotopy classes of complete hyperbolic metrics
of finite volume on $\Sigma_{g,n}[\sigma]$.  In other words, $\Teich[\sigma]$
is homeomorphic to the product of the Teichm\"{u}ller spaces of the components
of $\Sigma_{g,n}[\sigma]$.

These strata are pairwise disjoint and partition $\oTeich_{g,n}$.  In other words,
\[\oTeich_{g,n} = \bigsqcup_{\sigma} \oTeich[\sigma],\]
where the disjoint union is over $\emptyset$ together with all the simplices
of $\cC_{g,n}$.  The action of $\PMod_{g,n}$ preserves this stratification in the sense that
\[f(\Teich[\sigma]) = \Teich[f(\sigma)] \quad \text{for all $f \in \PMod_{g,n}$}.\]
In particular, the subgroup $\PMod\Span{\sigma}$ discussed in
\S \ref{section:stablesurface} takes $\Teich[\sigma]$
to itself.  The Dehn twists about the curves in $\sigma$ act
trivially on $\Teich[\sigma]$, so the action of $\PMod\Span{\sigma}$ on
$\Teich[\sigma]$ factors through $\PMod[\sigma]$.  

Let $S_1,\ldots,S_d$ be the components of $\Sigma_{g,n}[\sigma]$.  For each $1 \leq i \leq d$,
let $\PMod(S_i)$ be the
pure mapping class group of $S_i$ and let
$\Teich(S_i)$ be the Teichm\"{u}ller space of $S_i$.  We have
\begin{align*}
\PMod[\sigma] &= \PMod(S_1) \times \cdots \times \PMod(S_d),\\
\Teich[\sigma]      &= \Teich(S_1) \times \cdots \times \Teich(S_d).
\end{align*}
The action of $\PMod[\sigma]$ on $\Teich[\sigma]$ is the product of the actions
of $\PMod(S_i)$ on $\Teich(S_i)$.

\begin{remark}
From the fact that the Dehn twists about the curves in $\sigma$ act
trivially on $\Teich[\sigma]$, we see that except in the trivial case $(g,n) = (0,3)$,
the action of $\PMod_{g,n}$ on
$\oTeich_{g,n}$ is not properly discontinuous.  This is related to the
fact that $\oTeich_{g,n}$ is not locally compact.  
\end{remark}

\subsection{Closed strata}

For $\sigma$ either $\emptyset$ or a simplex of $\cC_{g,n}$, let 
$\oTeich[\sigma]$ be the closure of the stratum $\Teich[\sigma]$.  We have
\[\oTeich[\sigma] = \bigsqcup_{\tau \supseteq \sigma} \Teich[\tau],\]
where the disjoint union is over all collections of disjoint curves $\tau$ containing $\sigma$.
The action of $\PMod\Span{\sigma}$ on
$\oTeich_{g,n}$ preserves $\oTeich[\sigma]$, and just like for $\Teich[\sigma]$ its
action on $\oTeich[\sigma]$ factors through $\PMod[\sigma]$.  The closed
stratum is homeomorphic to the product of the augmented Teichm\"{u}ller spaces
of the components of $\Sigma_{g,n}[\sigma]$.  This homeomorphism is compatible
with the clutching map $\oModuli[\sigma] \rightarrow \Moduli_{g,n}$ in the sense
that we have a commutative diagram
\[\begin{tikzcd}
\oTeich[\sigma] \arrow{d}{/\PMod[\sigma]} \arrow[hook]{r} & \oTeich_{g,n} \arrow{d}{/\PMod_{g,n}} \\
\oModuli[\sigma] \arrow{r} & \Moduli_{g,n},
\end{tikzcd}\]
where the bottom arrow is the clutching map.

\subsection{Intersections of codimension-one strata}

For an essential simple closed curve $\gamma$, let $\Teich[\gamma] = \Teich[\{\gamma\}]$ and
$\oTeich[\gamma] = \oTeich[\{\gamma\}]$.  It follows from the discussion above
that the images of the closed strata $\oTeich[\gamma]$ in $\oModuli_{g,n}$ are the
boundary divisors.  For later use, we record the following observation about them:

\begin{lemma}
\label{lemma:intersectcurves}
Fix $g,n \geq 0$ such that $\Sigma_{g,n}$ is hyperbolic.  For some $k \geq 1$, let $\gamma_1,\ldots,\gamma_k$ be
pairwise distinct essential simple closed curves on $\Sigma_{g,n}$.
Set $\sigma = \{\gamma_1,\ldots,\gamma_k\}$.  Then the following hold:
\begin{itemize}
\item If $\sigma$ is a simplex of $\cC_{g,n}$, then
$\oTeich[\gamma_1] \cap \cdots \cap \oTeich[\gamma_k]$ is $\oTeich[\sigma]$, which is contractible.
\item If $\sigma$ is not a simplex of $\cC_{g,n}$,
then $\oTeich[\gamma_1] \cap \cdots \cap \oTeich[\gamma_k] = \emptyset$.
\end{itemize}
\end{lemma}
\begin{proof}
For each $1 \leq i \leq k$, we have
\[\oTeich[\gamma_i] = \bigsqcup_{\tau \supseteq \{\gamma_i\}} \Teich[\tau].\]
Here the union is over all simplices $\tau$ of $\cC_{g,n}$ containing $\gamma_i$.  It follows that
\[\oTeich[\gamma_1] \cap \cdots \cap \oTeich[\gamma_k] = \bigsqcup_{\tau \supseteq \sigma} \Teich[\tau].\]
Again, the union is over all simplices $\tau$ of $\cC_{g,n}$ containing $\sigma = \{\gamma_1,\ldots,\gamma_k\}$.
If $\sigma$ is not a simplex of $\cC_{g,n}$ then there are no such simplices $\tau$, so the
intersection is empty.  If $\sigma$ is a simplex of $\cC_{g,n}$, then it follows that
\[\oTeich[\gamma_1] \cap \cdots \cap \oTeich[\gamma_k] = \oTeich[\sigma].\]
The closed stratum $\oTeich[\sigma]$ is the product of the augmented Teichm\"{u}ller spaces of the components
of $\Sigma_{g,n}[\sigma]$.  Since each of these augmented Teichm\"{u}ller spaces is contractible,
it follows that $\oTeich[\sigma]$ is contractible.  The lemma follows.
\end{proof}

\section{Covers of the Deligne--Mumford compactification}
\label{section:compactcover}

Fix $g,n \geq 0$ such that $\Sigma_{g,n}$ is hyperbolic.  Let $G < \PMod_{g,n}$ be
a finite-index subgroup.  In \S \ref{section:modulicover}, we proved
that $\Teich_{g,n}/G$ is the analytification of an algebraic
variety $\Moduli_{g,n}(G)$ equipped with a finite map $\pi\colon \Moduli_{g,n}(G) \rightarrow \Moduli_{g,n}$
(see Theorem \ref{theorem:modulicoveralg}).  This section studies the analogue of the
Deligne--Mumford compactification for $\Moduli_{g,n}(G)$.

\subsection{Construction}

The ordinary Deligne--Mumford compactification $\oModuli_{g,n}$ is an irreducible
projective variety containing $\Moduli_{g,n}$ as a dense Zariski open subspace.
The variety $\Moduli_{g,n}(G)$ is irreducible and normal (see Lemma \ref{lemma:modulinormal}).
Recall that in \S \ref{section:normalization} we discussed normalization, which
is the following construction:
\begin{itemize}
\item Let $X$ be an irreducible algebraic variety and let $U \subset X$
be a nonempty Zariski open subset.  Let $p\colon \tU \rightarrow U$ be a finite
surjective map with $\tU$ normal and irreducible.  The {\em normalization} of
$X$ with respect to $p\colon \tU \rightarrow U$ is an irreducible normal algebraic
variety $\tX$ containing $\tU$ as a Zariski open subset together
with an extension of $p$ to a finite map $p\colon \tX \rightarrow X$.  This
always exists and is unique in an appropriate sense.
\end{itemize}
Applying this to our situation, we define $\oModuli_{g,n}(G)$
to be the normalization of $\oModuli_{g,n}$ with respect to the finite
map $\pi\colon \Moduli_{g,n}(G) \rightarrow \Moduli_{g,n}$.  The following lemma
summarizes some properties of $\oModuli_{g,n}(G)$:

\begin{lemma}
\label{lemma:constructdmg}
Fix $g,n \geq 0$ such that $\Sigma_{g,n}$ is hyperbolic.  Let $G < \PMod_{g,n}$ be
a finite-index subgroup.  Then $\oModuli_{g,n}(G)$ is a normal irreducible projective algebraic
variety containing $\Moduli_{g,n}(G)$ as a Zariski open dense subspace, and
the projection $\pi\colon \Moduli_{g,n}(G) \rightarrow \Moduli_{g,n}$ extends to a
finite map $\pi\colon \oModuli_{g,n}(G) \rightarrow \oModuli_{g,n}$.
\end{lemma}
\begin{proof}
Everything except for the fact that $\oModuli_{g,n}(G)$ is projective follows
from our discussion of normalizations in \S \ref{section:normalization}.  The
fact that $\oModuli_{g,n}(G)$ is projective follows from the fact that
$\oModuli_{g,n}$ is projective and there exists a finite map $\pi\colon \oModuli_{g,n}(G) \rightarrow \oModuli_{g,n}$
(see Theorem \ref{theorem:gaga}).
\end{proof}

\subsection{Analytic construction}

Recall that $\oModuli_{g,n}$ is the quotient of the augmented Teichm\"{u}ller space $\oTeich_{g,n}$
by $\PMod_{g,n}$.  Since $\oTeich_{g,n}$ is in general not a complex manifold or even a locally compact
topological space, it is not clear that $\oTeich_{g,n}/\PMod_{g,n}$ is anything other than just
a topological space.  However, Hubbard--Koch \cite{HubbardKoch} constructed a natural analytic
structure on $\oTeich_{g,n}/\PMod_{g,n}$ and proved that there is an analytic isomorphism
$\oTeich_{g,n}/\PMod_{g,n} \cong \oModuli_{g,n}$.

For a finite-index subgroup $G < \PMod_{g,n}$,
Hinich--Vaintrob proved that there is a similar description of $\oModuli_{g,n}(G)$.
In \cite[Theorem 6.1.1]{HinichVaintrob}, they proved the following:
\begin{itemize}
\item The space $\oTeich_{g,n}/G$ can be given an analytic structure.
\item With respect to this analytic structure, $\oTeich_{g,n}/G$ is an irreducible normal analytic variety.
\item The normal analytic variety $\oTeich_{g,n}/G$ contains $\Moduli_{g,n}(G) = \Teich_{g,n}/G$ as an open subvariety,
and the finite map $\Moduli_{g,n}(G) \rightarrow \Moduli_{g,n}$ extends to a finite
map $\oTeich_{g,n}/G \rightarrow \oModuli_{g,n}$.
\end{itemize}
Since $\oModuli_{g,n}$ is projective and $\oTeich_{g,n}/G \rightarrow \oModuli_{g,n}$ is a finite
map, it follows from Serre's GAGA theorems that $\oTeich_{g,n}/G$ is the analytification
of a projective algebraic variety (see Theorem \ref{theorem:gaga}).  This variety
is normal, so by the uniqueness of normalization it must be the case that there is an analytic
isomorphism between $\oTeich_{g,n}/G$ and $\oModuli_{g,n}(G)$.  This isomorphism fits into a commutative diagram
\[\begin{tikzcd}
\oTeich_{g,n}/G \arrow{r}{\cong} \arrow{d} & \oModuli_{g,n}(G) \arrow{d}{\pi} \\
\oTeich_{g,n}/\PMod_{g,n} \arrow{r}{\cong} & \oModuli_{g,n},
\end{tikzcd}\]
where the bottom isomorphism is the one constructed by Hubbard--Koch.

\subsection{Smoothness}

As was noted by Mostafa \cite{Mostafa} and van Geemen--Oort \cite{VanGeemenOort},
the variety $\oModuli_{g,n}(G)$ is in general not smooth even if $\Moduli_{g,n}(G)$ is smooth.
Singularities can occur along the boundary.  For $n=0$, Looijenga \cite{LooijengaSmooth}
proved that there exist finite-index normal subgroups $G \lhd \PMod_{g}$ such that
$\oModuli_{g}(G)$ is smooth.  This was extended to $n>0$ by
Boggi--Pikaart (\cite{BoggiPikaart}; see also \cite[\S 7.2.4]{HinichVaintrob}),
who proved the following:

\begin{theorem}[{Boggi--Pikaart \cite[Corollary 2.10]{BoggiPikaart}}]
\label{theorem:omodulismooth}
Fix $g,n \geq 0$ such that $\Sigma_{g,n}$ is hyperbolic.  Let
$\Gamma < \PMod_{g,n}$ be a finite-index subgroup.
Then there exists a finite-index subgroup $G < \PMod_{g,n}$ such that:
\begin{itemize}
\item[(i)] $G$ is a normal subgroup of $\PMod_{g,n}$ with $G \subset \Gamma$.
\item[(ii)] $\oModuli_{g,n}(G)$ is a smooth projective variety.
\end{itemize}
\end{theorem}

\begin{remark}
If $G_1,G_2 < \PMod_{g,n}$ are finite-index subgroups such that $G_1 < G_2$, then
there is a projection map $\oModuli_{g,n}(G_1) \rightarrow \oModuli_{g,n}(G_2)$.
Assume that $\oModuli_{g,n}(G_2)$ is smooth.  Experience with the open moduli
space might lead the reader to expect that $\oModuli_{g,n}(G_1)$ is smooth.  However,
the map $\oModuli_{g,n}(G_1) \rightarrow \oModuli_{g,n}(G_2)$ is typically not a covering
map, so we cannot deduce that $\oModuli_{g,n}(G_1)$ is smooth.  In fact, it need not be.
\end{remark}

\subsection{Boundary}

Define $\partial \oModuli_{g,n}(G) = \oModuli_{g,n}(G) \setminus \Moduli_{g,n}(G)$.
Boggi--Pikaart \cite{BoggiPikaart} proved the following.  In its statement, to
handle the trivial case $(g,n) = (0,3)$ where $\Moduli_{0,3}$ is a point we
regard the empty variety as a simple normal crossing variety.

\begin{theorem}[{Boggi--Pikaart \cite[Proposition 2.1]{BoggiPikaart}}]
\label{theorem:boundarysimple}
Fix $g,n \geq 0$ such that $\Sigma_{g,n}$ is hyperbolic.
Let $G \lhd \PMod_{g,n}$ be a finite-index normal subgroup such that the following hold:\footnote{Condition (a)
does not appear in the statement of \cite[Proposition 2.1]{BoggiPikaart}, but is needed to ensure
that the standing assumptions in \cite{BoggiPikaart} hold -- see \cite[p.\ 173]{BoggiPikaart}.}
\begin{itemize}
\item[(a)] For some $\ell \geq 3$, we have $G < \PMod_{g,n}(\ell)$.
\item[(b)] $\oModuli_{g,n}(G)$ is a smooth projective variety.
\end{itemize}
Then $\partial \oModuli_{g,n}(G)$ is a simple normal crossing variety.
\end{theorem}

\begin{remark}
The reference \cite[Proposition 2.1]{BoggiPikaart} only asserts
that each irreducible component of $\partial \oModuli_{g,n}(G)$ is
smooth, but as the first paragraph of its proof notes earlier
results in \cite{BoggiPikaart} imply that $\partial \oModuli_{g,n}(G)$ is
a normal crossing variety, so this is all that needs to be proved.
\end{remark}

\subsection{Clutching maps}

Let $\sigma$ be a simplex of the curve complex $\cC_{g,n}$.  Recall that
$\Sigma_{g,n}[\sigma]$ is the (possibly disconnected) topological surface obtained
by deleting the curves in $\sigma$ from $\Sigma_{g,n}$.  Let $S_1,\ldots,S_k$
be the components of $\Sigma_{g,n}[\sigma]$.  Each $S_i$ is a hyperbolic topological
surface, so we can talk about its augmented Teichm\"{u}ller space $\oTeich(S_i)$,
its pure mapping class group $\PMod(S_i)$,
and the Deligne--Mumford compactification $\oModuli(S_i)$ of its moduli space.
We then have
\begin{align*}
\oTeich[\sigma]  &= \oTeich(S_1) \times \cdots \times \oTeich(S_k),\\
\PMod[\sigma] &= \PMod(S_1) \times \cdots \times \PMod(S_k),\\
\oModuli[\sigma] &= \oModuli(S_1) \times \cdots \times \oModuli(S_k).
\end{align*}
The stabilizer subgroup $\PMod\Span{\sigma}$ discussed in \S \ref{section:stablesurface}
acts on $\oTeich[\sigma]$.  This action factors through the homomorphism
$\PMod\Span{\sigma} \rightarrow \PMod[\sigma]$ whose 
kernel is the subgroup of $\PMod\Span{\sigma}$ generated by the Dehn twists about
the curves in $\sigma$:
\[1 \rightarrow \SpanSet{$T_{\gamma}$}{$\gamma \in \sigma$} \rightarrow \PMod\Span{\sigma} \rightarrow \PMod[\sigma] \rightarrow 1.\]
In \S \ref{section:clutching1}, we discussed the algebraic clutching maps
$\oModuli[\sigma] \rightarrow \partial \oModuli_{g,n}$.  They are finite maps, so since $\oModuli[\sigma]$ is
an irreducible projective variety their images $\Delta[\sigma]$ are closed projective subvarieties
of $\partial \oModuli_{g,n}$.

In \cite[\S 7.3]{HinichVaintrob}, these clutching maps are extended to $\oModuli_{g,n}(G)$.
For $1 \leq i \leq k$, let $G_i < \PMod(S_i)$ be the intersection of $\PMod(S_i)$
with the image of $G \cap \PMod\Span{\sigma}$ in $\PMod[\sigma]$.  Set
\[G[\sigma] = G_1 \times \cdots \times G_k.\]
We remark that $G[\sigma]$
is contained in the image of $G \cap \PMod\Span{\sigma}$ in $\PMod[\sigma]$, but might
not equal that image.  The point here is that the image of $G \cap \PMod\Span{\sigma}$ in 
$\PMod[\sigma] = \PMod(S_1) \times \cdots \times \PMod(S_k)$ might not decompose
as a product of subgroups of the factors $\PMod(S_i)$.  Each $G_i$ is a finite-index subgroup
of $\PMod(S_i)$, so $G[\sigma]$ is a finite-index subgroup of $\PMod[\sigma]$.

Let $\oModuli(S_i,G_i)$ be the branched cover of $\oModuli(S_i)$ corresponding to $G_i$, so 
$\oModuli(S_i,G_i)$ is an irreducible projective algebraic variety and 
there is an analytic
isomorphism $\oTeich(S_i)/G_i \cong \oModuli(S_i,G_i)$.  Define
\[\oModuli[\sigma,G] = \oModuli(S_1,G_1) \times \cdots \times \oModuli(S_k,G_k).\]
This is an irreducible projective variety, and there is a clutching
map $\oModuli[\sigma,G] \rightarrow \partial \oModuli_{g,n}(G)$ fitting into a commutative
diagram
\[\begin{tikzcd}
\oTeich[\sigma] \arrow{d}{/G[\sigma]} \arrow{r} & \oTeich_{g,n} \arrow{d}{/G} \\
\oModuli[\sigma,G] \arrow{r} & \oModuli_{g,n}(G).
\end{tikzcd}\]
We have the following:

\begin{lemma}
\label{lemma:gclutchingfinite}
Fix $g,n \geq 0$ such that $\Sigma_{g,n}$ is hyperbolic.  Let $G < \PMod_{g,n}$ be a finite-index
subgroup and let $\sigma$ be a simplex of the curve complex $\cC_{g,n}$.  Then the clutching
map $\oModuli[\sigma,G] \rightarrow \oModuli_{g,n}(G)$ is a finite algebraic map.
\end{lemma}
\begin{proof}
It follows from its construction in \cite[\S 7.3]{HinichVaintrob} that the clutching map is analytic.
Since its domain and codomain are both
projective algebraic varieties, it is automatically algebraic (cf.\ the discussion
in \S \ref{section:algebravsanalysis}) and proper.  To prove that it is finite,
we must therefore prove that it has finite fibers (cf.\ \S \ref{section:finitemaps}).  
It fits into a commutative diagram
\[\begin{tikzcd}
\oModuli[\sigma,G] \arrow{d} \arrow{r} & \oModuli_{g,n}(G) \arrow{d} \\
\oModuli[\sigma] \arrow{r} & \oModuli_{g,n},
\end{tikzcd}\]
where the bottom map is the ordinary clutching map.  The ordinary clutching map
is finite, as are the two vertical projections in the above diagram.  We
conclude that the clutching map $\oModuli[\sigma,G] \rightarrow \oModuli_{g,n}(G)$ has
finite fibers, as desired.
\end{proof}

Since it is a finite map, the image of the clutching map
$\oModuli[\sigma,G] \rightarrow \partial \oModuli_{g,n}(G)$ is an irreducible projective
subvariety of $\partial \oModuli_{g,n}(G)$ (see \S \ref{section:finitemaps}).  We denote this image
by $\Delta[\sigma,G]$.

\subsection{Components of the boundary}

For an essential simple closed curve $\gamma$ on $\Sigma_{g,n}$, let $\Delta[\gamma,G] = \Delta[\{\gamma\},G]$,
so $\Delta[\gamma,G]$ is an irreducible projective subvariety 
of $\partial \oModuli_{g,n}(G)$.  If $\gamma$ and $\gamma'$ are essential simple closed
curves such that there exists some $f \in G$ with $f(\gamma) = \gamma'$, then
$\Delta[\gamma,G] = \Delta[\gamma',G]$.  In fact, the following holds:

\begin{lemma}
\label{lemma:gboundarydivisors}
Fix $g,n \geq 0$ such that $\Sigma_{g,n}$ is hyperbolic.
Let $G < \PMod_{g,n}$ be a finite-index subgroup and
let $\Lambda = \{\gamma(1),\ldots,\gamma(k)\}$ be a set containing exactly one representative
of each $G$-orbit of essential simple closed curves on $\Sigma_{g,n}$.
Then the irreducible components of $\partial \oModuli_{g,n}(G)$ are the $k$ sets
$\Delta[\gamma(1),G],\ldots,\Delta[\gamma(k),G]$.
\end{lemma}
\begin{proof}
Identical to the argument we gave in \S \ref{section:boundarydivisors} for identifying
the irreducible components of $\partial \oModuli_{g,n}$.
\end{proof}

\subsection{Dual complex of boundary}

Recall the following definition from \S \ref{section:dualcomplex}: 
\begin{itemize}
\item Let $X$ be a projective simple normal crossing variety
with irreducible components $X_1,\ldots,X_n$.  Define
\[X^{[p]} = \bigsqcup_{i_0 < \cdots < i_p} X_{i_0} \cap \cdots \cap X_{i_p} \quad \text{for $p \geq 0$}.\]
The {\em dual complex} of $X$ is the semisimplicial set whose set of $p$-simplices is $\pi_0(X^{[p]})$.  In other
words, the $p$-simplices of the dual complex of $X$ are the connected components of $X^{[p]}$
and the face maps correspond to inclusions of connected components.
\end{itemize}
The following strengthening of Theorem \ref{theorem:omodulismooth}
is the reason the moduli spaces $\Moduli_{g,n}(G)$ are
relevant to the topology of quotients of the curve complex $\cC_{g,n}$:

\begin{lemma}
\label{lemma:dualboundary}
Fix $g,n \geq 0$ such that $\Sigma_{g,n}$ is hyperbolic.  Let $\Gamma < \PMod_{g,n}$
be a finite-index subgroup.  
Then there exists a finite-index subgroup $G < \PMod_{g,n}$ such that:
\begin{itemize}
\item[(i)] $G$ is a torsion-free normal subgroup of $\PMod_{g,n}$ with $G \subset \Gamma$.
\item[(ii)] $\oModuli_{g,n}(G)$ is a smooth projective variety. 
\item[(iii)] $\partial \oModuli_{g,n}(G)$ is a simple normal crossing variety
whose dual complex\footnote{In the trivial case $(g,n) = (0,3)$ we have
$\partial \oModuli_{g,n}(G) = \emptyset$, and our convention is that the empty
set is a simple normal crossing variety whose dual complex is the empty set.}
 is homeomorphic to $\cC_{g,n}/G$.
\end{itemize}
\end{lemma}
\begin{proof}
Choose a finite-index normal subgroup $H < \PMod_{g,n}$ with
the following property:
\begin{itemize}
\item[$(\spadesuit)$] Let $\gamma$ be an essential simple closed curve on $\Sigma_{g,n}$.  For
some $f \in H$, assume that $f(\gamma)$ can be homotoped to be disjoint from $\gamma$.
Then $f(\gamma) = \gamma$.
\end{itemize}
It is well-known that such $H$ exist.  One reference for this
is \cite[Lemma 8.1]{BehrstockHagenMartinSisto}, which proves a much
stronger result.  We remark that the assumptions of \cite[Lemma 8.1]{BehrstockHagenMartinSisto} exclude
the cases $(g,n) \in \{(0,3),(0,4),(1,1)\}$, but in those cases any two disjoint essential simple closed curves
are homotopic, so we can just take $H = \PMod_{g,n}$.

Applying Theorem \ref{theorem:omodulismooth}, we can find a finite-index subgroup
$G < \PMod_{g,n}$ such that:
\begin{itemize}
\item[(i$'$)] $G$ is a normal subgroup of $\PMod_{g,n}$ with $G \subset \Gamma \cap H \cap \PMod_{g,n}(3)$.
\item[(ii)] $\oModuli_{g,n}(G)$ is a smooth projective variety.
\end{itemize}
Condition (i$'$) implies that $G \subset \Gamma$.  It also implies that $G \subset \PMod_{g,n}(3)$,
so by Lemma \ref{lemma:leveltorsionfree} the group $G$ is torsion-free.  Conclusion (i) therefore
holds.  Applying Theorem \ref{theorem:boundarysimple}, we get that
$\partial \oModuli_{g,n}(G)$ is a simple normal crossing variety.  This is half of (iii).
To prove the rest of (iii), we must prove that the dual complex of $\partial \oModuli_{g,n}(G)$
is homeomorphic to $\cC_{g,n}/G$.

By Lemma \ref{lemma:gboundarydivisors}, the irreducible components
of $\partial \oModuli_{g,n}(G)$ are exactly the $\Delta[\gamma,G]$
with $\gamma$ an essential simple closed curve on $\Sigma_{g,n}$.  Moreover, for 
essential simple closed curves $\gamma$ and $\gamma'$
we have $\Delta[\gamma,G] = \Delta[\gamma',G]$ if and only if
$\gamma$ and $\gamma'$ are in the same $G$-orbit.

Let $\sigma = \{\gamma_0,\ldots,\gamma_p\}$ be a set of
essential simple closed curves on $\Sigma_{g,n}$ that are in distinct $G$-orbits.  
We are not assuming that $\sigma$ is a simplex of $\cC_{g,n}$, though
$(\spadesuit)$ does imply that the curves in a simplex of $\cC_{g,n}$ are in distinct
$G$-orbits.  Say that a set $\sigma'$ of
essential simple closed curves on $\Sigma_{g,n}$ is a {\em $G$-rearrangement} of $\sigma$
if there exist $f_0,\ldots,f_p \in G$ such that
\[\sigma' = \{f_0(\gamma_0),\ldots,f_p(\gamma_p)\}.\]
Since the $\gamma_i$ lie in distinct $G$-orbits, the curves $f_i(\gamma_i)$ also
lie in distinct $G$-orbits.  This is an equivalence relation on sets of essential
simple closed curves lying in distinct $G$-orbits.
The heart of our proof is the following two facts:

\begin{claim}{1}
\label{claim:dualboundary.1}
Let $\{\gamma_0,\ldots,\gamma_p\}$ be a set of $(p+1)$ essential simple closed
curves on $\Sigma_{g,n}$ that lie in distinct $G$-orbits.
Then $\Delta[\gamma_0,G] \cap \cdots \cap \Delta[\gamma_p,G]$ is the union
of the set of all $\Delta[\sigma,G]$ such that $\sigma$ is a $p$-simplex of
$\cC_{g,n}$ that is a $G$-rearrangement of $\{\gamma_0,\ldots,\gamma_p\}$.  Each such
$\Delta[\sigma,G]$ is path-connected and closed.
\end{claim}
\begin{proof}[Proof of claim]
For $0 \leq i \leq p$, let $T_i$ be the preimage of $\Delta[\gamma_i,G]$ in
the augmented Teichm\"{u}ller space $\oTeich_{g,n}$.  The set $\Delta[\gamma_0,G] \cap \cdots \cap \Delta[\gamma_p,G]$
is therefore the image in $\oModuli_{g,n}(G)$ of $T_0 \cap \cdots \cap T_p$.
For $0 \leq i \leq p$, we have
\[T_i = \bigcup_{f \in G} \oTeich[f(\gamma_i)].\]
It follows that
\[T_0 \cap \cdots \cap T_p = \bigcup_{f_0,\ldots,f_p \in G} \oTeich[f_0(\gamma_0)] \cap \cdots \cap \oTeich[f_p(\gamma_p)].\]
For $f_0,\ldots,f_p \in G$, Lemma \ref{lemma:intersectcurves} says that
$\oTeich[f_0(\gamma_0)] \cap \cdots \cap \oTeich[f_p(\gamma_p)]$ is nonempty if and only if
$\sigma = \{f_0(\gamma_0),\ldots,f_p(\gamma_p)\}$ is a simplex of $\cC_{g,n}$, in
which case $\oTeich[f_0(\gamma_0)] \cap \cdots \cap \oTeich[f_p(\gamma_p)] = \oTeich[\sigma]$.
The set of all such $\sigma$ is the set of all $p$-simplices of $\cC_{g,n}$ that are $G$-rearrangements
of $\{\gamma_0,\ldots,\gamma_p\}$.
Letting $(\cC_{g,n})^{(p)}_{\gamma_0,\ldots,\gamma_p}$ be the set of $p$-simplices of $\cC_{g,n}$ that
are $G$-rearrangements of $\{\gamma_0,\ldots,\gamma_p\}$, we deduce that
\[T_0 \cap \cdots \cap T_p = \bigcup_{\sigma \in (\cC_{g,n})^{(p)}_{\gamma_0,\ldots,\gamma_p}}  \oTeich[\sigma].\]
Projecting this down to $\oModuli_{g,n}(G)$, we get that
\[\Delta[\gamma_0,G] \cap \cdots \cap \Delta[\gamma_p,G] = \bigcup_{\sigma \in (\cC_{g,n})^{(p)}_{\gamma_0,\ldots,\gamma_p}} \Delta[\sigma,G].\]
For $\sigma \in (\cC_{g,n})^{(p)}_{\gamma_0,\ldots,\gamma_p}$, we have that $\Delta[\sigma,G]$ is path-connected
since $\oTeich[\sigma]$ is contractible (see Lemma \ref{lemma:intersectcurves}).  Also, the subspace
$\Delta[\sigma,G]$ is closed since it is a projective subvariety of $\oModuli_{g,n}(G)$.  The claim follows.
\end{proof}

\begin{claim}{2}
\label{claim:dualboundary.2}
Let $\sigma$ and $\sigma'$ be $p$-simplices of $\cC_{g,n}$ such that $\sigma'$ is
a $G$-rearrangement of $\sigma$.  Then
$\Delta[\sigma,G]$ and $\Delta[\sigma',G]$
are disjoint unless $\sigma$ and $\sigma'$ are in the same $G$-orbit.
\end{claim}
\begin{proof}[Proof of claim]
Assume that $\Delta[\sigma,G] \cap \Delta[\sigma',G] \neq \emptyset$.  We will prove that
$\sigma'$ is in the same $G$-orbit as $\sigma$.
Write $\sigma = \{\gamma_0,\ldots,\gamma_p\}$.  For some $f_0,\ldots,f_p \in G$, we have
$\sigma' = \{f_0(\gamma_0),\ldots,f_p(\gamma_p)\}$.  The preimages of
$\Delta[\sigma,G]$ and $\Delta[\sigma',G]$ in $\oTeich_{g,n}$ are
\begin{align*}
\bigcup_{h \in G} \oTeich[h(\sigma)]  &= \bigcup_{h \in G} \oTeich[h(\gamma_0)] \cap \cdots \cap \oTeich[h(\gamma_p)], \\
\bigcup_{h \in G} \oTeich[h(\sigma')] &= \bigcup_{h \in G} \oTeich[h(f_0(\gamma_0))] \cap \cdots \cap \oTeich[h(f_p(\gamma_p))].
\end{align*}
For the right-hand sides of these, see Lemma \ref{lemma:intersectcurves}.  Since
$\Delta[\sigma,G] \cap \Delta[\sigma',G] \neq \emptyset$, there must exist some $h,h' \in G$ such that
\[\oTeich[h(\gamma_0)] \cap \cdots \cap \oTeich[h(\gamma_p)] \cap \oTeich[h'(f_0(\gamma_0))] \cap \cdots \cap \oTeich[h'(f_p(\gamma_p))] \neq \emptyset.\]
Applying Lemma \ref{lemma:intersectcurves}, we deduce that
\[\{h(\gamma_0),\ldots,h(\gamma_p),h'(f_0(\gamma_0)),\ldots,h'(f_p(\gamma_p))\}\]
is a simplex of $\cC_{g,n}$.  For $0 \leq i \leq p$, this implies in particular
that $h(\gamma_i)$ can be homotoped to be disjoint from $h'(f_i(\gamma_i))$.  We have
$h,h',f_i \in G$, and $G < H$ where $H$ satisfies $(\spadesuit)$:
\begin{itemize}
\item[$(\spadesuit)$] Let $\gamma$ be an essential simple closed curve on $\Sigma_{g,n}$.  For
some $f \in H$, assume that $f(\gamma)$ can be homotoped to be disjoint from $\gamma$.
Then $f(\gamma) = \gamma$.
\end{itemize}
We deduce that $h'(f_i(\gamma_i)) = h(\gamma_i)$.  This implies that
$h(\sigma) = h'(\sigma')$, so $\sigma$ and $\sigma'$ are in the same $G$-orbit, as desired.
\end{proof}

Recall that our goal is to prove that the dual complex of $\partial \oModuli_{g,n}(G)$
is homeomorphic to $\cC_{g,n}/G$.  Let $\Lambda = \{\gamma(1),\ldots,\gamma(k)\}$
be a set of orbit representatives for the $G$-action on the set of essential simple closed curves
on $\Sigma_{g,n}$.  Since $G < H$, the condition $(\spadesuit)$ implies two things:
\begin{itemize}
\item[$(\ast)$] The vertices of a simplex of $\cC_{g,n}$ lie in distinct $G$-orbits.
\item[$(\ast\ast)$] The group $G$ acts without rotations on
$\cC_{g,n}$, that is, the setwise stabilizer of a simplex stabilizes that simplex pointwise.
\end{itemize}
We conclude that $\cC_{g,n}/G$ is homeomorphic to the geometric realization of
the semisimplicial set $\fX$ whose $p$-simplices are the set of $G$-orbits of $p$-simplices
of $\cC_{g,n}$.  
To define the face maps on $\fX$, we 
choose compatible orderings on the vertices of simplices as follows.  Identify the
vertices of $\fX$ with $\Lambda = \{\gamma(1),\ldots,\gamma(k)\}$.  Order
the vertices of a simplex of $\fX$ using the ordering on $\Lambda$, and
define the $j^{\text{th}}$ face map by deleting the $j^{\text{th}}$ vertex.
For $1 \leq i_0 < \cdots < i_p \leq k$, the $p$-simplices of $\fX$ whose vertices
are $\{\gamma(i_0),\ldots,\gamma(i_p)\}$ are exactly the $G$-orbits
of $p$-simplices of $\cC_{g,n}$ that are $G$-rearrangements of $\{\gamma(i_0),\ldots,\gamma(i_p)\}$.

The $p$-simplices of the dual complex of $\partial \oModuli_{g,n}(G)$
are the set of path components of sets of the form
\[\Delta[\gamma(i_0),G] \cap \cdots \cap \Delta[\gamma(i_p),G] \quad \text{with $1 \leq i_0 < \cdots < i_p \leq k$}.\]
For $1 \leq i_0 < \cdots < i_p \leq k$, Claims \ref{claim:dualboundary.1} and \ref{claim:dualboundary.2} imply
that $\Delta[\gamma(i_0),G] \cap \cdots \cap \Delta[\gamma(i_p),G]$ is the disjoint union
of path-connected closed sets of the form $\Delta[\sigma,G]$ as $\sigma$ ranges over a set of orbit representatives
for the $G$-action on the set of $p$-simplices of $\cC_{g,n}$ that are $G$-rearrangements
of $\{\gamma(i_0),\ldots,\gamma(i_p)\}$.  There are finitely many such $G$-orbits,
so these are the path components of $\Delta[\gamma(i_0),G] \cap \cdots \cap \Delta[\gamma(i_p),G]$. 

Taking a path component
$\Delta[\sigma,G]$ to the $G$-orbit of $\sigma$ therefore gives a bijection between
the $p$-simplices of the dual complex and the $p$-simplices of $\fX$.
This bijection respects the corresponding face relations, and therefore implies that
these two complexes are isomorphic, as desired.
\end{proof}
 
\section{The proof of Boggi's theorem}
\label{section:boggiproof}

We now prove Boggi's theorem, whose statement we recall.  In it,
we emphasize that $G$ is a subgroup of $\Mod_{g,n}$, not the pure
mapping class group $\PMod_{g,n}$.

\newtheorem*{maintheorem:boggi}{Theorem~\ref{maintheorem:boggi}}
\begin{maintheorem:boggi}[{Boggi, cf.\ \cite[Lemma 5.5]{BoggiCongruence}}]
Fix $g,n \geq 0$.  Let $G < \Mod_{g,n}$ be a finite-index subgroup.
Then $\RH_k(\cC_{g,n}/G;\Q) = 0$ for $k \leq c(g,n)$, where
\[c(g,n) = \begin{cases}
2g-3 & \text{for $g \geq 1$ and $n = 0$},\\
2g+n-4 & \text{for $g \geq 1$ and $n \geq 1$},\\
n-5 & \text{for $g = 0$}.
\end{cases}\]
\end{maintheorem:boggi}
\begin{proof}
It is more convenient to work with cohomology, 
so we will prove that $\RH^k(\cC_{g,n}/G;\Q) = 0$
for $k \leq c(g,n)$.
Recall that Harer \cite[Theorem 3.5]{HarerVCD} proved that
$\cC_{g,n}$ is $c(g,n)$-connected.  Both being nonempty and being
path-connected pass to quotients, so there is nothing to prove
unless $c(g,n) \geq 1$, i.e., unless:
\begin{itemize}
\item $g \geq 2$; or
\item $g = 1$ and $n \geq 3$; or
\item $g = 0$ and $n \geq 6$.
\end{itemize}
These conditions ensure that $\Sigma_{g,n}$ is hyperbolic,
so our discussion of the moduli space of curves and its Deligne--Mumford compactification
in the previous sections applies.  Also, since being nonempty and
path-connected pass to quotients we must
only prove that $\HH^k(\cC_{g,n}/G;\Q) = 0$ for $1 \leq k \leq c(g,n)$.
We divide the proof into five steps.  The first two steps are preliminary.
The meat of the proof is in the last three steps.

\begin{step}{1}
\label{step:boggi.1}
We can assume without loss of generality that $G < \PMod_{g,n}$.
\end{step}

For the proof of this step, we temporarily rename $G$ to $\Gamma$.  
Set $G = \Gamma \cap \PMod_{g,n}$, so $G$ is a finite-index normal subgroup of $\Gamma$
with $G < \PMod_{g,n}$.
Assume that we can prove that $\HH^k(\cC_{g,n}/G;\Q) = 0$ for $1 \leq k \leq c(g,n)$.  We
will show that this will imply that $\HH^k(\cC_{g,n}/\Gamma;\Q) = 0$ for $1 \leq k \leq c(g,n)$,
so it is enough to prove the theorem for $G$.

Set $Q = \Gamma/G$.
The finite group $Q$ acts on $\cC_{g,n}/G$ and $\cC_{g,n}/\Gamma = (\cC_{g,n}/G)/Q$.
A standard result (see, e.g.,
\cite[Proposition 1.1]{PutmanNote}) now implies that for all $k \geq 0$ we have\footnote{The
reference \cite[Proposition 1.1]{PutmanNote} is about finite groups acting on simplicial complexes,
and $\cC_{g,n}/G$ might not inherit a natural simplicial complex structure from $\cC_{g,n}$.
However, by
subdividing $\cC_{g,n}$ appropriately we can ensure that $\cC_{g,n}/G$ does inherit a simplicial complex
structure from $\cC_{g,n}$, so \cite[Proposition 1.1]{PutmanNote} applies.}
\[\HH^k(\cC_{g,n}/\Gamma;\Q) \cong \HH^k(\cC_{g,n}/G;\Q)^{Q},\]
where the superscript means that we are taking the $Q$-invariants.
Since we are assuming that $\HH^k(\cC_{g,n}/G;\Q) = 0$ for $1 \leq k \leq c(g,n)$,
this implies that $\HH^k(\cC_{g,n}/\Gamma;\Q) = 0$ for $1 \leq k \leq c(g,n)$,
as desired.

\begin{step}{2}
\label{step:boggi.2}
We can assume without loss of generality that the following hold:
\begin{itemize}
\item[(i)] $G$ is a torsion-free normal subgroup of $\PMod_{g,n}$.
\item[(ii)] $\oModuli_{g,n}(G)$ is a smooth projective variety. 
\item[(iii)] $\partial \oModuli_{g,n}(G)$ is a simple normal crossing variety
whose dual complex is homeomorphic to $\cC_{g,n}/G$.
\end{itemize}
\end{step}

Again, temporarily rename $G$ to $\Gamma$.  Since
$\Gamma$ is a finite-index subgroup of $\PMod_{g,n}$,
Lemma \ref{lemma:dualboundary} implies that we can find
a finite-index normal subgroup $G$ of $\PMod_{g,n}$ such that
$G < \Gamma$ and such that (i)-(iii) hold.  Since $G$ is normal in $\PMod_{g,n}$,
it is also normal in $\Gamma$. 
The exact same argument we gave in Step \ref{step:boggi.1} shows
that to prove that $\HH^k(\cC_{g,n}/\Gamma;\Q) = 0$ for $1 \leq k \leq c(g,n)$,
it is enough to prove that $\HH^k(\cC_{g,n}/G;\Q) = 0$ for $1 \leq k \leq c(g,n)$.

\begin{step}{3}
\label{step:boggi.3}
We prove that $\HH^k_c(\Moduli_{g,n}(G);\Q) = 0$ for $k \leq c(g,n)+1$.  Here
we emphasize that $\Moduli_{g,n}(G)$ is the open moduli space and $\HH^k_c$ is compactly
supported cohomology.
\end{step}

Since $G$ is torsion-free, it follows that the moduli space $\Moduli_{g,n}(G) = \Teich_{g,n}/G$ is an
Eilenberg--Mac Lane space for $G$ (see Lemma \ref{lemma:coverproperties}).  
It also follows that the cohomological dimension of $G$
equals the virtual cohomological dimension of $\Mod_{g,n}$.
Harer \cite[Theorem 4.1]{HarerVCD} proved that the virtual
cohomological dimension of $\Mod_{g,n}$ is
\[\vcd(g,n) = \begin{cases}
4g-5   & \text{if $g \geq 2$ and $n=0$},\\
4g+n-4 & \text{if $g \geq 1$ and $n \geq 1$},\\
n-3    & \text{if $g = 0$ and $n \geq 3$}.
\end{cases}\]
Since $\Sigma_{g,n}$ is hyperbolic, the pair $(g,n)$ falls into one of these three cases.
It follows that $\HH_d(\Moduli_{g,n}(G);\Q) \cong \HH_d(G;\Q) = 0$ for $d \geq \vcd(g,n)+1$.
The space $\Moduli_{g,n}(G)$ is a smooth complex manifold with $\dim_{\C} \Moduli_{g,n}(G) = 3g+n-3$.
Its dimension as a (real) topological manifold is twice its complex dimension.  Applying
Poincar\'{e} duality, we see that
\[\HH^{(6g+2n-6)-d}_c(\Moduli_{g,n}(G);\Q) = 0 \quad \text{for $d \geq \vcd(g,n)+1$}.\]
In other words, $\HH^k_c(\Moduli_{g,n}(G);\Q) = 0$ for 
\[k \leq (6g+2n-6)-(\vcd(g,n)+1) = (6g+2n-7)-\vcd(g,n).\]
To complete the step, we must show that this vanishing range is equivalent to
$k \leq c(g,n)+1$.  Observe that
\[(6g+2n-7)-\vcd(g,n) = \begin{cases}
(6g-7)-(4g-5) = 2g-2 & \text{if $g \geq 2$ and $n=0$},\\
(6g+2n-7)-(4g+n-4) = 2g+n-3 & \text{if $g \geq 1$ and $n \geq 1$},\\
(2n-7)-(n-3) = n-4 & \text{if $g = 0$ and $n \geq 3$}.
\end{cases}\]
In all three cases, this is $c(g,n)+1$, as desired.

\begin{step}{4}
\label{step:boggi.4}
We prove that the inclusion $\partial \oModuli_{g,n}(G) \hookrightarrow \oModuli_{g,n}(G)$
induces a surjection 
\[\HH^k(\oModuli_{g,n}(G);\Q) \rightarrow \HH^k(\partial \oModuli_{g,n}(G);\Q) \quad \text{for $k \leq c(g,n)$}.\]
\end{step}

The long exact sequence of the pair $(\oModuli_{g,n}(G),\partial \oModuli_{g,n}(G))$ contains
the segment
\[\begin{tikzcd}[row sep=small, column sep=small]
\HH^k(\oModuli_{g,n}(G);\Q) \arrow{r} & \HH^k(\partial \oModuli_{g,n}(G);\Q) \arrow{r} & \HH^{k+1}(\oModuli_{g,n}(G), \partial \oModuli_{g,n}(G);\Q).
\end{tikzcd}\]
To prove the step, we must prove that the third term in this sequence vanishes for $k \leq c(g,n)$.
Letting $\Moduli_{g,n}(G)^{\ast}$ be the one-point compactification of $\Moduli_{g,n}(G)$, we have
\begin{align*}
\HH^{k+1}(\oModuli_{g,n}(G),\partial \oModuli_{g,n}(G);\Q) \cong &\RH^{k+1}(\oModuli_{g,n}(G)/\partial \oModuli_{g,n}(G);\Q) \\
\cong &\RH^{k+1}(\Moduli_{g,n}(G)^{\ast};\Q) \cong \HH^{k+1}_c(\Moduli_{g,n}(G);\Q).
\end{align*}
Step \ref{step:boggi.3} says that this vanishes for $k \leq c(g,n)$.  The step follows.

\begin{step}{5}
\label{step:boggi.5}
We prove that $\HH^k(\cC_{g,n}/G;\Q) = 0$ for $1 \leq k \leq c(g,n)$.
\end{step}

Since $\oModuli_{g,n}(G)$ is a smooth projective variety, Lemma \ref{lemma:strictpreserve} implies
that the image of the map $\HH^k(\oModuli_{g,n}(G);\Q) \rightarrow \HH^k(\partial \oModuli_{g,n}(G);\Q)$
is pure of weight $k$.  Step \ref{step:boggi.4} therefore implies that
$\HH^k(\partial \oModuli_{g,n}(G);\Q)$ is pure of weight $k$ for $k \leq c(g,n)$.  Being pure of weight $k$
implies that the weight-$(k-1)$ subspace vanishes, so this implies
in particular that
\[W_0 \HH^k(\partial \oModuli_{g,n}(G);\Q) = 0 \quad \text{for $1 \leq k \leq c(g,n)$}.\]
Theorem \ref{theorem:weight0} says that $W_0 \HH^k(\partial \oModuli_{g,n}(G);\Q)$ is
isomorphic to the $k^{\text{th}}$ rational cohomology group of the dual complex of
the simple normal crossing variety $\partial \oModuli_{g,n}(G)$, which is
homeomorphic to $\cC_{g,n}/G$.  The step follows.
\end{proof}

\section{Tool use}

Before they will allow me to post this paper, the moderators of the arXiv have told
me to declare any use of ``sophisticated tools'' in its preparation.  I suspect that
they do not consider me a ``sophisticated tool'' for the purposes of this declaration.
Despite this, I hope the following remarks will satisfy them:
\begin{itemize}
\item I have understood the technical details of the proofs in this paper for well over
a decade.  My 2012 paper \cite{PutmanWieland} depended on Theorem \ref{maintheorem:boggi},
and I would not have been comfortable writing \cite{PutmanWieland} without knowing
how to prove Theorem \ref{maintheorem:boggi}.
\item Every word of this paper was written by me on my iMac using {\tt vim} in a terminal window.
\item The figures were all drawn by hand using the software {\tt Inkscape}.  I did
use an LLM to convert the labels on the figures from the {\tt Inkscape} default
font to Computer Modern so they would match the font used in the rest of the paper.
\item The references are all papers/books that I already knew about or found using mathscinet
and google.  I suppose I might have looked at the AI overview at the top of a google search.
\item I used an LLM to check the grammar and usage of this paper after it was written.
Because of this, it has slightly fewer ``the the''s, and also takes more care
with trivial edge cases.  For instance, the LLM was very insistent that
since $\Moduli_{0,3}$ is a single point it is wrong to say that $\Moduli_{g,n}$ is
not compact and that we need to add points corresponding to stable curves to compactify it.
\item To find references in SGA, I used the LLM generated English translation of SGA that was
posted at \url{https://grothendiecksga.com}.  I would like to thank the anonymous
creators of that website.
\end{itemize}

\end{document}